\documentclass[12pt,english]{article}
\usepackage[T1]{fontenc}
\usepackage[latin9]{inputenc}
\usepackage{amsthm}
\usepackage{amsmath}
\usepackage{amssymb}
\usepackage{graphicx}
\usepackage{setspace}
\usepackage{esint}
\makeatletter
\numberwithin{equation}{section}
\numberwithin{figure}{section}
\theoremstyle{plain}
\newtheorem{thm}{\protect\theoremname}
  \theoremstyle{plain}
  \newtheorem{prop}{\protect\propositionname}
 \theoremstyle{definition}
  \newtheorem{example}{\protect\examplename}

\author{Sheng Yu and Enrique Campos-N\'{a}\~{n}ez\\ \normalsize Dept. of Engineering Management and System Engineering\\ \normalsize The George Washington University\\}
\date{Originally written in Spring, 2011}

\@ifundefined{showcaptionsetup}{}{%
 \PassOptionsToPackage{caption=false}{subfig}}
\usepackage{subfig}
\makeatother

\usepackage{babel}
  \providecommand{\examplename}{Example}
  \providecommand{\propositionname}{Proposition}
\providecommand{\theoremname}{Theorem}

\begin{document}

\title{Computing Supply Function Equilibria via Spline Approximations}
\maketitle
\begin{abstract}
The supply function equilibrium (SFE) is a model for competition in
markets where each firm offers a schedule of prices and quantities
to face demand uncertainty, and has been successfully applied to wholesale
electricity markets. However, characterizing the SFE is difficult,
both analytically and numerically. In this paper, we first present
a specialized algorithm for capacity constrained asymmetric duopoly
markets with affine costs. We show that solving the first order conditions
(a system of differential equations) using spline approximations is
equivalent to solving a least squares problem, which makes the algorithm
highly efficient. We also propose using splines as a way to improve
a recently introduced general algorithm, so that the equilibrium can
be found more easily and faster with less user intervention. We show
asymptotic convergence of the approximations to the true equilibria
for both algorithms, and illustrate their performance with numerical
examples.
\end{abstract}

\section{Introduction}

With the presence of demand uncertainty, firms may choose to compete
in supply functions -- a schedule of prices and quantities that correspond
to different realizations of the demand. Such concept was first introduced
by Klemperer and Meyer \cite{klemperer1989supply}, and they named
the non-cooperative Nash Equilibrium of this type of games the Supply
Function equilibrium (SFE). Soon, people found that the competition
in the deregularized wholesale electricity markets bear high resemblance
to this formulation, and Green and Newbery \cite{green1992competition}
first applied this model to the England and Wales market. Since then,
modeling behaviors of wholesale electricity markets has been an important
application of the SFE.

The SFE model has attracted tremendous attention from both industry
and academia. Despite its popularity, people found the SFE model difficult
due to the following issues: (1) The first order necessary conditions
of the SFE is a system of ordinary differential equations (ODE), shown
in Klemperer and Meyer \cite{klemperer1989supply}, but when people
try to solve this system of ODEs, they usually find that the solutions
are not increasing functions%
\footnote{In this paper, the terms ``increasing'' and ``non-decreasing''
are used interchangeably. Supply functions must be non-decreasing
(increasing) to be feasible, but not necessarily strictly increasing.%
}, which is a requirement for feasibility; (2) there can be an infinite
number of supply function equilibria, leading to an equilibrium selection
problem; (3) it is hard to incorporate capacity constraints and general
cost functions to the framework, and allowing supply functions to
have a free form makes the problem even more complicated, therefore
many studies are limited to symmetric firms and/or restraining the
solution space to functions of simple forms, such as linear%
\footnote{We use ``linear'' for first-degree polynomials, which are also called
affine functions in some literature.%
} and quadratic functions.

Despite all the difficulties, researchers have made substantial progress
both in theoretical and computational analysis of the SFE. Klemperer
and Meyer \cite{klemperer1989supply} provided foundational analysis
of the supply function equilibrium, and compares and contrasts the
SFE with the equilibria of Cournot and Bertrand games. It also showed
the existence of the SFE of symmetric oligopolies (assuming convex
costs, a concave demand, and no capacity constraints), and showed
that there could usually be infinite supply function equilibria, unless
the support of the distribution of the demand shock was unbounded.

Later, researchers found that capacity constraints could greatly reduce
the number of potential supply function equilibria, and sometimes
even make it unique. Holmberg \cite{holmberg2008unique} proved that
if we had symmetric producers, inelastic demand, a price cap and if
the capacity constraints bound with a positive probability, then we
had a unique symmetric equilibrium. With the same conditions, except
that the producers had identical marginal costs but asymmetric capacities,
Holmberg \cite{holmberg2007supply} showed the equilibrium was unique
and piecewise symmetric.

Perhaps the most important topic in the study of SFEs is how to find
them. Finding the SFE is difficult, and restrictions are usually needed.
Rudkevich et al.~\cite{rudkevich1998modeling} provided a closed-form
formula for cases where demand was inelastic and firms did not have
capacity constraints, and it further showed that the equilibrium price
had a high mark-up compared to the perfectly competitive price. Green
\cite{green1996increasing,green1999electricity} restricted the supplies
to linear functions, and applied the model to the England and Wales
market. Baldick et al.~\cite{baldick2004theory} showed how to find
linear and piecewise linear SFE when the demand and marginal costs
were linear.

A popular approach to finding the SFE is to work on the first order
conditions (a system of ODEs). Many of the studies following this
approach involved the use of numerical integration, but the major
difficulty was that the initial conditions were unknown, and without
the right initial conditions, the integrals so calculated would usually
not be increasing functions, i.e. they were not feasible supply functions.
With the assumption that capacity constraints of smaller firms bind
earlier, Holmberg \cite{holmberg2009numerical} provided a procedure
for solving the ODE system via numerical integration that searched
for feasible solutions by tuning the initial conditions, i.e., the
prices at which the capacity constraints were reached.

Baldick and Hogan \cite{baldick2001capacity} proposed an alternative
approach that used an iterative scheme for finding the SFE: At each
step, each firm updates its supply function by moving to the best
response to the other firms' previous offers with a discount factor.
This procedure was repeated and hopefully it would converge. However,
as the authors pointed out, the computational cost (finding the best
response) of this iterative scheme was huge, even when it converged.

To use the iterative scheme, one must first know how to find the optimal
response to a given set of supplies. Anderson and Philpott \cite{anderson2002using}
provided conditions for the existence of the optimal response, and
analyzed the bound of difference in profit when approximations of
the supply functions were needed. Anderson and Philpott \cite{anderson2002optimal}
and Anderson and Xu \cite{anderson2002necessary} expressed the expected
return of the firms as line integrals, proved the existence of the
optimal supply function, and gave necessary and sufficient conditions
for optimality. Rudkevich \cite{rudkevich2003supply} described an
algorithm for developing piecewise linear optimal responses by cutting
the $x$-$y$ plane into blocks.

Baldick and Hogan \cite{baldick2004polynomial} discussed using high
degree polynomials in the iterative scheme as a parametric form of
the supply functions. The authors pointed out that such approximation
was not stable. By ``stable'' they meant that given a small perturbation
in the equilibrium, the supply functions would still converge to the
same equilibrium if the firms followed the iterative scheme.

Anderson and Hu \cite{anderson2008finding} showed conditions for
the SFEs' continuity and differentiability, which served as theoretical
guides to algorithm development. In addition, they proposed a numerical
method for finding the SFE that allowed the firms to have heterogeneous
capacities and costs. Their method allowed the supply functions to
have free form, and approximates them with piecewise linear functions.
To find the equilibrium, the method searched for a feasible solution
by solving an auxiliary nonlinear program (NLP) that had the necessary
conditions as the constraints. This method has been successfully applied
in Sioshansi and Oren \cite{sioshansi2007good}, which showed some
large generators in the ERCOT electricity market in Texas bid approximately
in accordance with the SFE model.

All the above studies of supply function equilibria focused on continuous
supply functions, while in practice offers in most markets are step-functions.
Holmberg et al.~\cite{holmberg2008supply} showed convergence of
the discrete SFE to the well studied continuous one as the number
of steps increases.

In this paper we benefit from Anderson and Hu \cite{anderson2008finding},
and focus on numerical methods for finding the SFE. We allow the supply
functions to have free form, and we will exploit the capability of
splines to approximate the SFE accurately.

In the first part of this paper, we provide a specialized algorithm
for markets of asymmetric duopolies that have constant marginal costs.
In Section \ref{sec:SFE-without-cap}, we express the first order
conditions with splines, reduce the problem of solving the system
of ODEs to a least squares problem, and show that the solution space
of the ODE system and that of the least squares problem are the same
(in terms of approximation). Since we avoided using numerical integration,
we do not have the initial point selection problem. The solution of
the least squares problem has a very simple form, and we show in Section
\ref{sec:twoGenCap} that we can use a linear search to find the unique
equilibrium of the capacitated problem. In Section \ref{sec:general method},
we propose an improvement to the general method given by Anderson
and Hu \cite{anderson2008finding}. We will see that with the use
of splines, the number of decision variables and constraints can be
greatly reduced, thus in principle solving the optimization problem
should be easier and faster. Uniform convergence will be shown for
both methods. Examples that demonstrate the use and the properties
of these methods are provided in Section \ref{sec:NumericalExamples}.

\section{Solving the First Order Necessary Conditions\label{sec:SFE-without-cap}}

Our model considers a market with $m$ firms. Each firm $i$ has a
maximal capacity $Cap_{i}$. Let $C_{i}(q)$ be the cost of firm $i$
for producing an amount of $q$. Assume $C_{i}(q)$ is convex, non-decreasing
and differentiable for all $i$. Each firm knows the exact cost function
and capacity of its own, as well as those of all the other competitors.

The market demand is a function of the form $D(p,\varepsilon)=D(p)+\varepsilon$.
$D(p)$ is strictly decreasing, continuously differentiable and concave,
and it is known to all firms. The demand shock $\varepsilon$ is a
continuous random variable, and all the firms know that $\varepsilon$
has positive probability density on $[\varepsilon_{min},\varepsilon_{max}]$.
We will focus on the type of SFE that each point of the supply function
is the best response to a realization of $\varepsilon$, also termed
as ``strong SFE'' in Anderson and Hu \cite{anderson2008finding},
so the knowledge of the exact distribution of $\varepsilon$ is not
necessary for finding the equilibrium.

The supply function of firm $i$ is a non-decreasing function $s_{i}:[0,p_{max}]\rightarrow[0,Cap_{i}]$,
where $p_{max}=\sup\{p\ge0\mid D(p,\varepsilon_{max})\ge0\}$. If
there is a market specified price cap and if it is less than $\sup\{p\ge0\mid D(p,\varepsilon_{max})\ge0\}$,
let $p_{max}$ equal to the price cap.

As first pointed out in Klemperer and Meyer \cite{klemperer1989supply},
in an SFE, the supply functions $\{s_{j}\}_{j=1}^{m}$ must maximize
each firm's profit 
\begin{equation}
\max_{p}\; p\left[D(p)+\varepsilon-\sum_{i\ne j}s_{i}(p)\right]-C_{j}(D(p)+\varepsilon-\sum_{i\ne j}s_{i}(p)),\: j=1,\dots,m\label{eq:maxProfit}
\end{equation}
at all $p\in[0,p_{max}]$. If the supply functions $\{s_{j}\}_{j=1}^{m}$
are differentiable at $p$, and if $0<s_{j}(p)<Cap_{j}$ for all $j$,
then we have the first order conditions 
\[
\sum_{i\ne j}s_{i}'(p)-\frac{s_{j}(p)}{p-C_{j}'\left(s_{j}(p)\right)}=D'(p),\: j=1,\dots,m.
\]

Throughout Section \ref{sec:SFE-without-cap} and \ref{sec:twoGenCap},
we assume that the marginal costs are constant. Let the marginal cost
for firm $j$ be $c_{j}$. Then the first order conditions reduce
to 
\begin{equation}
\sum_{i\ne j}s_{i}'(p)-\frac{s_{j}(p)}{p-c_{j}}=D'(p),\: j=1,\dots,m.\label{eq:firstOrder}
\end{equation}

Anderson and Hu \cite{anderson2008finding} proves that in an equilibrium,
the supply functions are continuous for $p\notin\left\{ c_{1},\dots,c_{m}\right\} $.
Furthermore, it shows that in an equilibrium, each supply function
$s_{i}(p)$ is continuously differentiable at $p\notin\left\{ c_{1},\dots,c_{m}\right\} \cup\left\{ p_{Cap_{1}},\dots,p_{Cap_{m}}\right\} $,
where $p_{Cap_{i}}$ is the price where firm $i$ reaches its capacity,
i.e., $p_{Cap_{i}}=\inf\{p\mid s_{i}(p)=Cap_{i}\}$. Assume $\max(c_{1},\dots,c_{m})<\min(p_{Cap_{1}},\dots,p_{Cap_{m}})$.
Since the supply functions are increasing, if $0<s_{i}(p)<Cap_{i}$
is true for all $i$ at a price $p$, then we must have $\max(c_{1},\dots,c_{m})<p<\min(p_{Cap_{1}},\dots,p_{Cap_{m}})$,
which implies that $\{s_{i}\}_{i=1}^{m}$ are continuously differentiable
at $p$. Therefore, $\{s_{i}\}_{i=1}^{m}$ is a solution to the ODE
system (\ref{eq:firstOrder}) for all the prices $p$ such that $\max(c_{1},\dots,c_{m})<p<\min(p_{Cap_{1}},\dots,p_{Cap_{m}})$.
Since we do not know the value of $\min(p_{Cap_{1}},\dots,p_{Cap_{m}})$
yet, we will solve (\ref{eq:firstOrder}) numerically for $p\in\left(p_{min},p_{max}\right)$,
where $p_{min}=\max(c_{1},\dots,c_{m})$, and we will find $\min(p_{Cap_{1}},\dots,p_{Cap_{m}})$
in the next section.

Since $\{s_{i}(p)\}_{i=1}^{m}$ are continuously differentiable on
$(\max(c_{1},\dots,c_{m}),\min(p_{Cap_{1}},\dots,p_{Cap_{m}}))$,
it is a good idea to approximate them with splines. To achieve continuous
differentiability, the splines we use should be at least of order
3 (quadratic splines). Order 4 splines (cubic splines) are prefered
by most people, as they are the lowest-order splines that are smooth
to human eyes.

Splines have been very popular for their capability for approximation.
And beginning from the late 1960's, splines are being used by mathematicians
to develop numerical solutions to ordinary and partial differential
equations. We will fundamentally do the same in this section. To estimate
the spline coefficients, one can either use interpolation or use least
squares estimation. In this paper we use the latter one, and the reason
will be justified shortly.

We start by selecting knots for the spline approximation, and for
simplicity, we will let the knots for all the supply functions be
the same, as this is good enough according to our numerical experience.
Then, according to the type of splines we use, we will have basis
functions associated with the knots. Denote the bases with $B_{t}(x)$,
$t=1,\dots,K$, $K$ depending on the type and the order of the splines.
Let $B(x)=(B_{1}(x),\dots,B_{K}(x))^{T}$, a vector of basis functions.
Denote the spline approximation of $s_{i}(p)$ with 
\[
\hat{s}_{i}(p)=\sum_{t}b_{it}B_{t}(p)=B^{T}(p)\cdot\beta_{i},
\]
where $b_{it}$ are the coefficients to be determined and $\beta_{i}=(b_{i1},\dots,b_{iK})^{T}$
is the coefficient vector for $\hat{s}_{i}$.

We replace the supply functions in the first order conditions (\ref{eq:firstOrder})
with their spline approximations. The equations now become 
\begin{equation}
\sum_{i\ne j}\sum_{t}b_{it}B_{t}'(p)-\frac{\sum_{t}b_{jt}B_{t}(p)}{p-c_{j}}=D'(p),\: j=1,\dots,m.\label{eq:firstOrderSpline}
\end{equation}
Observe that with $p$ fixed, (\ref{eq:firstOrderSpline}) is linear
in $b_{it}$, $i=1,\dots,m$, $t=1,\dots,K$. Thus (\ref{eq:firstOrderSpline})
can also be written in matrix form: 
\begin{equation}
\mathfrak{B}_{\{j\}}^{T}(p)\cdot\beta=D'(p),\: j=1,\dots,m,\label{eq:firstOrderSplineMatrix}
\end{equation}
where 
\[
\mathfrak{B}_{\{j\}}(p)=(\underbrace{B'^{T}(p),\dots,B'^{T}(p)}_{j-1},-\frac{B^{T}(p)}{p-c_{j}},\underbrace{B'^{T}(p),\dots,B'^{T}(p)}_{m-j})^{T},
\]
is a vector of functions, and 
\[
\beta=(\beta_{1}^{T},\dots,\beta_{m}^{T})^{T}.
\]

The necessary conditions (\ref{eq:firstOrderSplineMatrix}) are linear
equations that the splines are expected to satisfy. Hence it is natural
to use the least squares method to estimate the coefficients, which
is part of the reason of our choice. At each price $p$, (\ref{eq:firstOrderSplineMatrix})
provides $m$ equations. We have $K\cdot m$ coefficients to estimate,
thus one may wish to choose at least $K$ prices from the range $\left(p_{min},p_{max}\right)$
to determine $b_{it}$, $i=1,\dots,m$, $t=1,\dots,K$.%
\footnote{In fact as we will see very soon, it is not enough to determine the
coefficients. But to reduce confusion, let us just proceed at this
point.%
} Let the selected prices be $p_{1},...,p_{N}\in\left(p_{min},p_{max}\right)$.
Let 
\[
\mathbb{B}=(\mathfrak{B}_{\{1\}}(p_{1}),\dots,\mathfrak{B}_{\{m\}}(p_{1}),\dots,\mathfrak{B}_{\{1\}}(p_{N}),\dots,\mathfrak{B}_{\{m\}}(p_{N}))^{T},
\]
and 
\[
d=(\underbrace{D'(p_{1}),\dots,D'(p_{1})}_{m},\dots,\underbrace{D'(p_{N}),\dots,D'(p_{N})}_{m})^{T}.
\]
We expect the spline approximations to satisfy the linear system $\mathbb{B}\beta=d$,
where a typical line of the system, say, $\mathfrak{B}_{\{j\}}^{T}(p_{k})\beta=D'(p_{k})$,
is a characterization of the relationship between $\hat{s}_{j}$ and
the derivatives of all the other supply functions at price $p_{k}$.
To estimate $\{b_{it}\}$, we solve the optimization problem 
\begin{equation}
\min_{\{b_{it}\}}\left\Vert \mathbb{B}\beta-d\right\Vert ^{2}.\label{eq:LSE}
\end{equation}

Before we solve this minimization problem, we would like to have a
look at the solution to the original ODE system analytically.

Consider a market with two firms 1 and 2. (\ref{eq:firstOrder}) is
now a set of two equations: 
\begin{align*}
s_{1}'(p)-\frac{s_{2}(p)}{p-c_{2}} & =D'(p),\\
s_{2}'(p)-\frac{s_{1}(p)}{p-c_{1}} & =D'(p),
\end{align*}
whose homogeneous problem 
\begin{align*}
s_{1}'(p)-\frac{s_{2}(p)}{p-c_{2}} & =0,\\
s_{2}'(p)-\frac{s_{1}(p)}{p-c_{1}} & =0,
\end{align*}
has solution 
\begin{align*}
s_{1}(p) & =t_{1}(p-c_{1})+\frac{t_{2}}{(c_{1}-c_{2})^{2}}\left(c_{2}-c_{1}+(p-c_{1})\log\left(\frac{p-c_{2}}{p-c_{1}}\right)\right),\\
s_{2}(p) & =t_{1}(p-c_{2})+\frac{t_{2}}{(c_{1}-c_{2})^{2}}\left(c_{2}-c_{1}+(p-c_{2})\log\left(\frac{p-c_{2}}{p-c_{1}}\right)\right),
\end{align*}
where $t_{1},t_{2}\in\mathbb{R}$, that is, a homogeneous solution
is a linear combination of two fundamental solutions. However, if
$t_{2}\ne0$, it is easy to verify that as $p\rightarrow\max(c_{1},c_{2})$
from above, either $\vert s_{1}\vert\rightarrow\infty$ or $\vert s_{2}\vert\rightarrow\infty$.
Therefore for the practical background of our problem, $t_{2}$ must
be 0, and consequently the homogeneous solution is just
\begin{align}
s_{1}(p) & =t(p-c_{1}),\nonumber \\
s_{2}(p) & =t(p-c_{2}).\label{eq:homogeneous solution}
\end{align}
Thus if $\{s_{1}^{0}(p),s_{2}^{0}(p)\}$ and $\{s_{1}^{1}(p),s_{2}^{1}(p)\}$
are two equilibria, we must have $s_{1}^{1}(p)=s_{1}^{0}(p)+t(p-c_{1})$,
$s_{2}^{1}(p)=s_{2}^{0}(p)+t(p-c_{2})$ for some $t$. On the other
hand, if $\{s_{1}(p),s_{2}(p)\}$ is a solution to the ODE system
(\ref{eq:firstOrder}), then $\{s_{1}(p)+t(p-c_{1}),s_{2}(p)+t(p-c_{2})\}$
is a solution, too, for any $t\in\mathbb{R}$. Thus (\ref{eq:firstOrder})
has infinite solutions, and we next show that our spline approximation
can indeed represent all these solutions in duopoly markets. This
result holds unless $\mathbb{B}$ has columns of zeros%
\footnote{This happens when an knot interval contains no $p_{k}$ if we use
B-splines. If we use natural cubic splines, it happens if neither
of the last two knot intervals contains any $p_{k}$.%
} or has fewer rows than columns.
\begin{thm}
\textup{\label{P: singularity}} For duopoly markets, $\mathbb{B}$
does not have full column rank. Furthermore, the rank of $\mathbb{B}$
is $2K-1$, unless it contains columns of zeros or has fewer rows
than columns. These results do not dependent on the type of splines
and the selection of knots.\end{thm}
\begin{proof}
Write the first order conditions in matrix form 
\begin{align*}
B'^{T}(p)\beta_{1}-\frac{B^{T}(p)\beta_{2}}{p-c_{2}} & =D'(p),\\
-\frac{B^{T}(p)\beta_{1}}{p-c_{1}}+B'^{T}(p)\beta_{2} & =D'(p).
\end{align*}
To prove the first part of Theorem \ref{P: singularity}, it is sufficient
to show that the matrix of functions 
\[
\left(\begin{array}{cc}
B'^{T}(p) & -\frac{B^{T}(p)}{p-c_{2}}\\
-\frac{B^{T}(p)}{p-c_{1}} & B'^{T}(p)
\end{array}\right)
\]
has linearly dependent columns. And since elementary row operations
preserve rank, it is equivalent to show that the columns of 
\[
\left(\begin{array}{cc}
(c_{2}-p)B'^{T}(p) & B^{T}(p)\\
B^{T}(p) & (c_{1}-p)B'^{T}(p)
\end{array}\right)
\]
are linearly dependent. We prove this by using the fact that the elements
of $B(p)$ form a basis of the space $\mathbb{S}$, which is composed
of all the splines on $(p_{min},p_{max})$ with the prescribed order
and knots.

Assume that we have a nonzero vector $v^{T}=(v_{1}^{T},v_{2}^{T})$
such that 
\[
\left(\begin{array}{cc}
(c_{2}-p)B'^{T}(p) & B^{T}(p)\\
B^{T}(p) & (c_{1}-p)B'^{T}(p)
\end{array}\right)\left(\begin{array}{c}
v_{1}\\
v_{2}
\end{array}\right)=0,
\]
or equivalently
\begin{align*}
(c_{2}-p)B'^{T}(p)v_{1}+B^{T}(p)v_{2} & =0,\\
B^{T}(p)v_{1}+(c_{1}-p)B'^{T}(p)v_{2} & =0.
\end{align*}
Let $f_{1}(p)=B^{T}(p)v_{1}$ and $f_{2}(p)=B^{T}(p)v_{2}$. Thus
the above can be rewritten as
\begin{align}
(c_{2}-p)f_{1}'(p)+f_{2}(p) & =0,\nonumber \\
f_{1}(p)+(c_{1}-p)f_{2}'(p) & =0.\label{eq:singularity}
\end{align}
 Recall that $f_{1}(p)$ and $f_{2}(p)$ are splines. Thus (\ref{eq:singularity})
implies that $f_{1}(p)$ and $f_{2}(p)$ must be linear functions
(they are single-piece linear functions because of their smoothness).
Therefore, we must have $f_{1}(p)=t(p-c_{1})$ and $f_{2}(p)=t(p-c_{2})$,
where $t$ is an arbitrary scalar.

Since $f_{1}(p)=t(p-c_{1})\in\mathbb{S}$ and $f_{2}(p)=t(p-c_{2})\in\mathbb{S}$,
$v_{1}$ and $v_{2}$ must exist and are unique. If $t\ne0$, then
we have $v_{1}\ne0$ and $v_{2}\ne0$, thus we proved that $\mathbb{B}$
does not have full rank.

Furthermore, since $t(p-c_{2})$ and $t(p-c_{1})$ are the only forms
that $f_{1}(p)$ and $f_{2}(p)$ can have, it implies that the null
space of $\mathbb{B}$ has only one dimension, i.e., the rank of $\mathbb{B}$
is $2K-1$.
\end{proof}
Theorem \ref{P: singularity} shows that when we have only two firms,
the general solution to the optimization problem ``$\text{minimize}\left\Vert \mathbb{B}\beta-d\right\Vert ^{2}$''
has the form $\beta=\beta^{0}+t\cdot v$, where $t\in\mathbb{R}$
and where $v$ is an eigenvector of $\mathbb{B}^{T}\mathbb{B}$ whose
corresponding eigenvalue is 0. In terms of individual supply functions,
the general solutions are $\hat{s}_{1}=B^{T}(p)(\beta_{1}^{0}+tv_{1})=B^{T}(p)\beta_{1}^{0}+t(p-c_{1})$
and $\hat{s}_{2}=B^{T}(p)(\beta_{2}^{0}+tv_{2})=B^{T}(p)\beta_{2}^{0}+t(p-c_{2})$,
which have the same form as the analytical solutions, showing that
the splines are able to approximate all the solutions. This is the
most important reason why we use least squares for the estimation
of the coefficients. When the market has three or more firms, Theorem
\ref{P: singularity} no longer holds --- $\mathbb{B}$ will generally
have full rank, and consequently the optimal solution will be unique
and does not represent the solution space of the original ODE system.

We would also like to show the asymptotic property of this spline
approximation. We will take cubic splines as an illustration. Proofs
for other types of splines are essentially the same.
\begin{thm}
\label{thm:ODE_limit}If $s_{1}$ and $s_{2}$ are continuously differentiable
on $[p_{a},p_{b}]$, where $p_{a}>\max(c_{1},c_{2})$, and if $\hat{s}_{1}$
and $\hat{s}_{2}$ are piecewise cubic splines and are solution to
the least squares problem (\ref{eq:LSE}), where we place $N$ price
levels $\{p_{k}\}_{k=1}^{N}$ uniformly%
\footnote{This is in fact unnecessary. We place the price levels this way solely
for making the Riemann integral easier to write.%
} among $[p_{a},p_{b}]$ and choose $\{\tau_{t}\}_{t=1}^{K}$ as the
knots, then as the length of the largest knot interval $\vert\tau\vert\rightarrow0$
and $N\rightarrow\infty$, the ODE system (\ref{eq:firstOrderSpline})
will be satisfied by $\hat{s}_{1}$ and $\hat{s}_{2}$, such that
the error functions, $\hat{s}_{-i}'(p)-\frac{\hat{s}_{i}(p)}{p-c_{i}}-D'(p)$,
$i=1,2$, uniformly converge to 0 on $[p_{a},p_{b}]$.%
\footnote{$\hat{s}_{i}$ is determined by $N$ and the knots $\{\tau_{t}\}$,
so it is more rigorous to write $\hat{s}_{i}(p\mid N,\{\tau_{t}\})$.
However, we will proceed with $\hat{s}_{i}$ for conciseness.%
}\end{thm}
\begin{proof}
Since $\hat{s}_{1}$, $\hat{s}_{2}$ and $D'$ are continuous on $[p_{a},p_{b}]$,
and since $p_{a}>\max(c_{1},c_{2})$, the sum of squared errors 
\[
\sum_{i=1,2}\left(\hat{s}_{-i}'(p)-\frac{\hat{s}_{i}(p)}{p-c_{i}}-D'(p)\right)^{2}
\]
is Riemann integrable on $[p_{a},p_{b}]$. We will show that the integral
\[
\int_{p_{a}}^{p_{b}}\left[\sum_{i=1,2}\left(\hat{s}_{-i}'(p)-\frac{\hat{s}_{i}(p)}{p-c_{i}}-D'(p)\right)^{2}\right]dp\rightarrow0,
\]
as $\vert\tau\vert\rightarrow0$ and $N\rightarrow\infty$.

Since $(\hat{s}_{1},\hat{s}_{2})$ is a solution to the least squares
problem, $\hat{s}_{1}$ and $\hat{s}_{2}$ minimize 
\[
\sum_{k=1}^{N}\left[\sum_{i=1,2}\left(\hat{s}_{-i}'(p_{k})-\frac{\hat{s}_{i}(p_{k})}{p_{k}-c_{i}}-D'(p_{k})\right)^{2}\right]
\]
among all the functions in $\mathbb{S}$, the space that contains
all the piecewise cubic splines with the selected knots $\{\tau_{t}\}$.
Let $I_{4}s_{1}$ and $I_{4}s_{2}$ denote the complete cubic interpolation
of $s_{1}$ and $s_{2}$ with knots $\{\tau_{t}\}$.%
\footnote{Similar to $\hat{s}_{i}$, it is actually $I_{4}s_{i}(p\mid\{\tau_{t}\})$,
but for conciseness we will use $I_{4}s_{i}$.%
} Since $I_{4}s_{1},I_{4}s_{2}\in\mathbb{S}$, we must have 
\begin{align}
 & \sum_{k=1}^{N}\left[\sum_{i=1,2}\left(\hat{s}_{-i}'(p_{k})-\frac{\hat{s}_{i}(p_{k})}{p_{k}-c_{i}}-D'(p_{k})\right)^{2}\right]\nonumber \\
\le & \sum_{k=1}^{N}\left[\sum_{i=1,2}\left(I_{4}s_{-i}'(p_{k})-\frac{I_{4}s_{i}(p_{k})}{p_{k}-c_{i}}-D'(p_{k})\right)^{2}\right].\label{eq:errorUB1}
\end{align}

There are upper bounds for the error of the complete cubic interpolations.
For example, from de Boor \cite{de2001practical} one knows that for
any $g\in C^{1}[a,b]$, we have $\left\Vert I_{4}g-g\right\Vert _{\infty}\le\frac{57}{8}\left|\tau\right|\omega(g',\frac{\left|\tau\right|}{2})$
and $\left\Vert I_{4}g'-g'\right\Vert _{\infty}\le\frac{57}{4}\omega(g',\frac{\left|\tau\right|}{2})$,
where $\left\Vert f\right\Vert _{\infty}=\sup\{\left|f(x)\right|\mid x\in[a,b]\}$
and $\omega(f,h)=\sup\{\left|f(x)-f(y)\right|\mid\left|x-y\right|<h\}$.
So there exists a constant $C$, such that for any $g\in C^{1}[a,b]$,
$\left\Vert I_{4}g-g\right\Vert _{\infty}\le C\left|\tau\right|\omega(g',\frac{\left|\tau\right|}{2})$
and $\left\Vert I_{4}g'-g'\right\Vert _{\infty}\le C\omega(g',\frac{\left|\tau\right|}{2})$.

Therefore as $\vert\tau\vert\rightarrow0$, for any $p\in[p_{a},p_{b}]$,
\begin{align}
 & \left(I_{4}s_{-i}'(p)-\frac{I_{4}s_{i}(p)}{p-c_{i}}-D'(p)\right)^{2}\nonumber \\
= & \left((I_{4}s_{-i}'(p)-s_{-i}'(p))-\left(\frac{I_{4}s_{i}(p)}{p-c_{i}}-\frac{s_{i}(p)}{p-c_{i}}\right)+s_{-i}'(p)-\frac{s_{i}(p)}{p-c_{i}}-D'(p)\right)^{2}\nonumber \\
= & \left((I_{4}s_{-i}'(p)-s_{-i}'(p))-\left(\frac{I_{4}s_{i}(p)}{p-c_{i}}-\frac{s_{i}(p)}{p-c_{i}}\right)\right)^{2}\nonumber \\
\le & C^{2}\left(\omega(s_{-i}',\frac{\left|\tau\right|}{2})+\frac{\left|\tau\right|\omega(s_{i}',\frac{\left|\tau\right|}{2})}{(p_{a}-c_{i})}\right)^{2}\rightarrow0,\label{eq:errorUB2}
\end{align}
where the second equality is because $s_{-i}'(p)-\frac{s_{i}(p)}{p-c_{i}}=D'(p)$,
$i=1,2$, a necessary condition for an SFE, and the inequality is
by applying the error bounds and the triangle inequality. Convergence
is due to uniform continuity of $s_{i}'$, $i=1,2$, on $[p_{a},p_{b}]$.

So by using the definition of Riemann integral, we have
\begin{align*}
0\le & \lim_{\vert\tau\vert\rightarrow0}\int_{p_{a}}^{p_{b}}\left[\sum_{i=1,2}\left(\hat{s}_{-i}'(p)-\frac{\hat{s}_{i}(p)}{p-c_{i}}-D'(p)\right)^{2}\right]dp\\
= & \lim_{\vert\tau\vert\rightarrow0}\lim_{N\rightarrow\infty}\left[\frac{p_{b}-p_{a}}{N}\sum_{k=1}^{N}\left[\sum_{i=1,2}\left(\hat{s}_{-i}'(p_{k})-\frac{\hat{s}_{i}(p_{k})}{p_{k}-c_{i}}-D'(p_{k})\right)^{2}\right]\right]\\
\le & \lim_{\vert\tau\vert\rightarrow0}\lim_{N\rightarrow\infty}\left[\frac{p_{b}-p_{a}}{N}\sum_{k=1}^{N}\left[\sum_{i=1,2}\left(I_{4}s_{-i}'(p_{k})-\frac{I_{4}s_{i}(p_{k})}{p_{k}-c_{i}}-D'(p_{k})\right)^{2}\right]\right]\\
\le & \lim_{\vert\tau\vert\rightarrow0}\lim_{N\rightarrow\infty}\frac{p_{b}-p_{a}}{N}N\left[\sum_{i=1,2}C^{2}\left(\omega(s_{-i}',\frac{\left|\tau\right|}{2})+\frac{\left|\tau\right|\omega(s_{i}',\frac{\left|\tau\right|}{2})}{(p_{a}-c_{i})}\right)^{2}\right]\\
= & \lim_{\vert\tau\vert\rightarrow0}(p_{b}-p_{a})\left[\sum_{i=1,2}C^{2}\left(\omega(s_{-i}',\frac{\left|\tau\right|}{2})+\frac{\left|\tau\right|\omega(s_{i}',\frac{\left|\tau\right|}{2})}{(p_{a}-c_{i})}\right)^{2}\right]\\
= & 0,
\end{align*}
where the second inequality is by (\ref{eq:errorUB1}), and the third
inequality and the convergence are by (\ref{eq:errorUB2}).

The uniform convergence follows naturally as $[p_{a},p_{b}]$ is closed
and the error functions $\hat{s}_{-i}'(p)-\frac{\hat{s}_{i}(p)}{p-c_{i}}-D'(p)$,
$i=1,2$, are continuous on $[p_{a},p_{b}]$.

When $s_{1}$ and $s_{2}$ satisfy stronger conditions, we can use
tighter bounds. For instance, Hall \cite{hall1968error} and Hall
and Meyer \cite{hall1976optimal} show that if $g\in C^{4}[a,b]$,
then the tightest bounds are $\left\Vert I_{4}g-g\right\Vert _{\infty}\le\frac{5}{384}\left|\tau\right|^{4}\Vert g^{(4)}\Vert_{\infty}$
and $\left\Vert I_{4}g'-g'\right\Vert _{\infty}\le\frac{1}{24}\left|\tau\right|^{3}\Vert g^{(4)}\Vert_{\infty}$.
\end{proof}
Now we have a simple form of the spline approximations, which we know
will converge to the true solutions as the mesh of the splines becomes
finer. In the next section we will take this advantage and find the
SFE with capacity constraints.

\section{SFE of Duopolies with Capacity Constraints\label{sec:twoGenCap}}

In Section \ref{sec:SFE-without-cap} we solved the necessary conditions
for duopoly markets. In this section, we still focus on duopoly markets,
and we use the solutions from Section \ref{sec:SFE-without-cap} to
find the SFE with capacity constraints. SFE of more firms will be
discussed in the next section.

In the following, we denote by $f(x^{-})$ and $f(x^{+})$ the left
and right limits of $f$, respectively. Similarly, we denote by $f'(x^{-})$
and $f'(x^{+})$ the left and right derivatives of $f$, respectively.
Same as in Section \ref{sec:SFE-without-cap}, we use $p_{Cap_{i}}$
for the price where $s_{i}$ reaches the capacity, i.e., $p_{Cap_{i}}=\inf\{p\mid s_{i}(p)=Cap_{i}\}$.
\begin{prop}
\label{P:smooth_i}In a 2-firm-SFE, assume firm 1 reaches the capacity
earlier than firm 2 does, i.e., $p_{Cap_{1}}<p_{Cap_{2}}$. Also assume
that $s_{2}(p_{Cap_{1}})>0$. Then $s_{1}$ is differentiable at $p_{Cap_{1}}$,
and the derivative is 0. In other words, the supply function that
reaches the capacity first must reach it smoothly.\end{prop}
\begin{proof}
Since and $s_{2}(p_{Cap_{1}})>0$, we have $\max(c_{1},c_{2})<p_{Cap_{1}}<p_{Cap_{2}}$
(see Anderson and Hu \cite{anderson2008finding}). So there exists
$\delta>0$, such that $s_{1}(p)$ is differentiable for $p\in(p_{Cap_{1}}-\delta,p_{Cap_{1}}+\delta)\backslash\{p_{Cap_{1}}\}$.
And when $s_{1}(p)$ is differentiable, the first order condition
(\ref{eq:firstOrder}) for $s_{2}(p)$ can be written as 
\begin{equation}
s_{2}(p)=(p-c_{2})(s_{1}'(p)-D'(p)).\label{eq:firstOrderKai}
\end{equation}

Once $s_{1}$ reaches $Cap_{1}$, it cannot decrease, as we require
supply functions to be non-decreasing. Thus $s_{1}(p)=Cap_{1}$ for
$p\ge p_{Cap_{1}}$, and $s_{1}'(p_{Cap_{1}}^{+})=0$. If $s_{1}$
were not smooth at $p_{Cap_{1}}$, i.e. $s_{1}'(p_{Cap_{1}}^{-})>0$,
then from (\ref{eq:firstOrderKai}), we must have 
\begin{eqnarray*}
s_{2}(p_{Cap_{1}}^{-}) & =(p_{Cap_{1}}-c_{2})(s_{1}'(p_{Cap_{1}}^{-})-D'(p_{Cap_{1}}))\\
 & >(p_{Cap_{1}}-c_{2})(s_{1}'(p_{Cap_{1}}^{+})-D'(p_{Cap_{1}})) & =s_{2}(p_{Cap_{1}}^{+}).
\end{eqnarray*}
Thus $s_{2}$ would be decreasing at $p_{Cap_{1}}$, and it would
be disqualified as a supply function. Therefore, we must have $s_{1}'(p_{Cap_{1}}^{-})=s_{1}'(p_{Cap_{1}}^{+})=0$.\end{proof}
\begin{prop}
\label{P:nonsmooth_j}Under the same assumptions of Proposition \ref{P:smooth_i},
if $D(p)$ is twice differentiable, then the derivative of $s_{2}$
has a jump at $p_{Cap_{1}}$, i.e., $s_{2}'(p_{Cap_{1}}^{+})>s_{2}'(p_{Cap_{1}}^{-})$.\end{prop}
\begin{proof}
Differentiate both sides of (\ref{eq:firstOrderKai}), we have 
\[
s_{2}'(p)=(s_{1}'(p)-D'(p))+(p-c_{2})(-D''(p)+s_{1}''(p)).
\]
Since $s_{1}''(p_{Cap_{1}}^{-})<0$, $s_{1}''(p_{Cap_{1}}^{+})=0$,
and Proposition \ref{P:smooth_i} shows that $s_{1}'(p_{Cap_{1}})=0$,
the left and right limits of $s_{2}'$ must have the relationship
\begin{eqnarray*}
s_{2}'(p_{Cap_{1}}^{-}) & =-D'(p_{Cap_{1}})+(p_{Cap_{1}}-c_{2})(s_{1}''(p_{Cap_{1}}^{-})-D''(p_{Cap_{1}}))\\
 & <-D'(p_{Cap_{1}})+(p_{Cap_{1}}-c_{2})(s_{1}''(p_{Cap_{1}}^{+})-D''(p_{Cap_{1}})) & =s_{2}'(p_{Cap_{1}}^{+}).
\end{eqnarray*}

\end{proof}
Propositions \ref{P:smooth_i} and \ref{P:nonsmooth_j} help us understand
the nature of the SFE with capacity constraints. Proposition \ref{P:smooth_i}
also provides a hint on how to find the equilibrium.

Recall from Section \ref{sec:SFE-without-cap} that a general solution
can be written as $\beta=\beta^{0}+t\cdot v$, where $t\in\mathbb{R}$
and where $v$ is the eigenvector of $\mathbb{B}^{T}\mathbb{B}$ that
corresponds to the eigenvalue 0. Note that $\beta$ is just a solution
to the ODE system, and $B^{T}(p)\cdot\beta_{i}$, $i=1,2$ may well
be decreasing or even negative at some part of $(p_{min},p_{max})$.
Our aim is to find the $\beta$, by adjusting $t$, such that $B^{T}(p)\cdot\beta_{1}$
is a nondecreasing curve, and has maximum $Cap_{1}$ at a price $p>p_{min}$,
which we define as $p_{Cap_{1}}$, and that $B^{T}(p)\cdot\beta_{2}$
is nondecreasing from $p_{min}$ to $p_{Cap_{1}}$ with $B^{T}(p_{Cap_{1}})\cdot\beta_{2}<Cap_{2}$.
(Swap 1 and 2 if necessary.) If $p_{Cap_{1}}\ge p_{max}$, then the
capacities are not binding. If $p_{Cap_{1}}<p_{max}$, then $Cap_{1}$
is binding, and the estimated equilibrium will be 
\[
\hat{s}_{1}(p)=\begin{cases}
\max(-D'(p)(p-c_{1}),0), & p\le p_{min};\\
B^{T}(p)\cdot\beta_{1}, & p_{min}<p\le p_{Cap_{1}};\\
Cap_{1}, & p_{Cap_{1}}<p<p_{max};
\end{cases}
\]
and 
\[
\hat{s}_{2}(p)=\begin{cases}
\max(-D'(p)(p-c_{2}),0), & p\le p_{min};\\
B^{T}(p)\cdot\beta_{2}, & p_{min}<p\le p_{Cap_{1}};\\
\min(-D'(p)(p-c_{2}),Cap_{2}), & p_{Cap_{1}}<p<p_{max}.
\end{cases}
\]

Monotonicity is not an issue at the lower end where $p\le p_{min}=\max(c_{1},c_{2})$.
In this price range, the firm with the higher marginal cost does not
produce, and the one with the lower marginal cost outputs at the monopolistic
level. When $p$ reaches $\max(c_{1},c_{2})$, the high cost firm
begins to produce, and (\ref{eq:firstOrderKai}) tells us that the
low cost firm can only have a sudden increase in supply at that price.
Hence there is no issue with monotonicity.

In practice, finding the appropriate $\beta=\beta^{0}+t\cdot v$ is
easy. By plotting the splines, we can easily spot the trend of how
the supply functions change when we adjust $t$, and we can also see
intuitively that in general there can be at most one SFE with capacities
constraints (sometimes an SFE just does not exist). If we are convinced
that an SFE exists, we just need to do a linear search (thanks to
the 1-dimensional solution space) to find the $t$ that makes one
of the supply curves reaches its capacity smoothly, according to Proposition
\ref{P:smooth_i}.

Theorem \ref{thm:ODE_limit} guarantees that in the limit situation,
$\hat{s}_{1}$ and $\hat{s}_{2}$ are solution to the ODE system (\ref{eq:firstOrderSpline})
for $p\in(p_{min},p_{Cap_{1}})$, and by Proposition 3 in Holmber
et al.~\cite{holmberg2008supply}, the $\hat{s}_{1}$ and $\hat{s}_{2}$
so constructed are indeed a supply function equilibrium.
\begin{prop}
\label{P:uniqueness}If $D(p_{min})+\varepsilon_{min}<0$, then there
can be at most one (strong) SFE with capacity constraints.\end{prop}
\begin{proof}
The condition $D(p_{min})+\varepsilon_{min}<0$ means that it is possible
that the demand is sometimes really low, and the market clearing price
must be lower than $p_{min}=\max(c_{1},c_{2})$, thus by its definition,
an SFE has to include prices below $p_{min}$, which further means
that when $p\in(p_{min},\min(p_{Cap_{1}},p_{Cap_{2}}))$, the difference
between two equilibria has to be $t(p-c_{1})$ for $s_{1}(p)$ and
$t(p-c_{2})$ for $s_{2}(p)$, for some $t\in\mathbb{R}$, according
to (\ref{eq:homogeneous solution}). Without loss of generality, assume
$s_{1}(p)$ reaches its capacity first, at $p_{Cap_{1}}$. Proposition
\ref{P:smooth_i} shows that $s_{1}'(p_{Cap_{1}})=0$. If $\bar{s}_{1}(p)$
is a supply function of firm 1 in any equilibrium, we must have $\bar{s}_{1}(p)=s_{1}(p)+t(p-c_{1})$
for some $t$. If $t$ was positive, then $\bar{s}_{1}(p)$ would
reach its capacity at a price $\bar{p}<p_{Cap_{1}}$. Since $s_{1}'(\bar{p})\ge0$,
we have $\bar{s}_{1}'(\bar{p})=s_{1}'(\bar{p})+t>0$, which contradicts
with Proposition \ref{P:smooth_i}. If $t$ was negative, then we
have $\bar{s}_{1}(p)=s_{1}(p)+t(p-c_{1})<Cap_{1}$ for $p\le p_{Cap_{1}}$,
and $\bar{s}_{1}'(p_{Cap_{1}})=s_{1}'(p_{Cap_{1}})+t=0+t<0$, which
disqualifies $\bar{s}_{1}(p)$ as a supply function. Therefore $t$
has to be 0, which means $\bar{s}_{1}=s_{1}$, and hence we cannot
have two distinct SFE.
\end{proof}
All we discuss in this paper are strong SFEs, which are not guaranteed
to exist. However, Anderson \cite{anderson2011supply} shows that
at least for duopoly markets, weak SFEs always exist, which is beyond
the discussion of this paper.

\section{A General Method for Finding SFE\label{sec:general method}}

When the market has more than two firms, the solution to the least
squares problem will be unique. Thus the method used in Section \ref{sec:twoGenCap}
for finding the SFE with capacity constraints will not work, and hence
we need a new method. Also, we would like a method that handles general
cost functions, instead of just linear ones. But first of all, we
would like to show some conditions that an SFE must satisfy when we
have more than two firms.

\subsection{Properties at the Nonsmooth Points in Multiplayer SFE\label{sub:Difficulty}}

In Section \ref{sec:twoGenCap} we saw that when we have two firms,
Proposition \ref{P:smooth_i} shows that the supply function that
reaches its capacity first must reach it smoothly. When we have $m$
firms, $m>2$, for the same reason, the $m$-1th supply function to
reach its capacity should still reach it smoothly, but the first $m-2$
supply functions do not have to.

In Anderson and Hu \cite{anderson2008finding} the authors show that
in an SFE, a supply curve can be discontinuous only at a price where
another firm begins to produce, while all the other producing firms
are at their capacities. This leads to the following consequences,
which should be observed in a good SFE approximation.

Suppose $p_{Cap_{1}}$ is the price where $s_{1}$ reaches its capacity
$Cap_{1}$, and suppose that firms $1\dots n$, $n>2$, are producing
at $p_{Cap_{1}}^{-}$ and are not bound by their capacities. Then
due to continuity, the following must hold:
\begin{enumerate}
\item $s_{i}(p_{Cap_{1}})=s_{i}(p_{Cap_{1}}^{-})=(p_{Cap_{1}}-c_{i})(\sum_{j\ne i}s_{j}'(p_{Cap_{1}}^{-})-D'(p_{Cap_{1}})),\: i=1,\dots,n;$
\item $s_{i}(p_{Cap_{1}})=s_{i}(p_{Cap_{1}}^{+})=(p_{Cap_{1}}-c_{i})(\sum_{j\ne1,i}s_{j}'(p_{Cap_{1}}^{+})-D'(p_{Cap_{1}})),\: i\ne1.$
\end{enumerate}
And together they imply 
\[
s_{1}'(p_{Cap_{1}}^{-})=\sum_{i\ne1,2}\left(s_{i}'(p_{Cap_{1}}^{+})-s_{i}'(p_{Cap_{1}}^{-})\right)=\cdots=\sum_{i\ne1,n}\left(s_{i}'(p_{Cap_{1}}^{+})-s_{i}'(p_{Cap_{1}}^{-})\right),
\]
which further implies 
\[
s_{2}'(p_{Cap_{1}}^{+})-s_{2}'(p_{Cap_{1}}^{-})=\cdots=s_{n}'(p_{Cap_{1}}^{+})-s_{n}'(p_{Cap_{1}}^{-})=\frac{s_{1}'(p_{Cap_{1}}^{-})}{n-2}.
\]
It means that the rest of the curves $s_{i}$, $i\ne1$ are not differentiable
at $p_{Cap_{1}}$, and the right limits of their derivatives minus
the left limits are all equal. Graphically, in an SFE, we expect to
see all these curves have a jump in their derivatives at this price
by the same amount.

At the lower price level, where firms begin to produce, similar things
happen, but only that it is now a decrease in the derivatives: At
$c_{1}$, firm 1 begins to produce. And suppose that firms $1\dots n$,
$n>2$, are producing at $c_{1}^{+}$and are not bound by their capacities.
Then we must have:
\begin{enumerate}
\item $s_{i}(c_{1})=s_{i}(c_{1}^{+})=(c_{1}-c_{i})(\sum_{j\ne i}s_{j}'(c_{1}^{+})-D'(c_{1}^{+})),\: i=1,\dots,n;$
\item $s_{i}(c_{1})=s_{i}(c_{1}^{-})=(c_{1}-c_{i})(\sum_{j\ne1,i}s_{j}'(c_{1}^{-})-D'(c_{1}^{-})),\: i\ne1.$
\end{enumerate}
Together they imply 
\[
-s_{1}'(c_{1}^{+})=\sum_{i\ne1,2}\left(s_{i}'(c_{1}^{+})-s_{i}'(c_{1}^{-})\right)=\cdots=\sum_{i\ne1,n}\left(s_{i}'(c_{1}^{+})-s_{i}'(c_{1}^{-})\right),
\]
which further implies 
\[
s_{2}'(c_{1}^{+})-s_{2}'(c_{1}^{-})=\cdots=s_{n}'(c_{1}^{+})-s_{n}'(c_{1}^{-})=-\frac{s_{1}'(c_{1}^{+})}{n-2}.
\]
In a graph of the SFE, the already producing firms will have a drop
in their derivatives by the same amount, whenever there is a new firm
begins production.

\subsection{A General Method}

In this subsection we develop a general method that works for markets
with arbitrary number of players, and the cost functions are no longer
assumed to be linear. Of course this general method can work with
duopolies with linear cost functions, but still the method introduced
in Sections \ref{sec:SFE-without-cap} and \ref{sec:twoGenCap} are
recommended, as least squares problems are extremely easy to solve.

In Anderson and Hu \cite{anderson2008finding} the authers show how
to use piecewise linear functions to approximate the supply functions
in an equilibrium. They list the necessary conditions that the supply
functions of an SFE must satisfy, and try to find a set of piecewise
linear functions that satisfy these conditions at selected prices.
To do so, they form an auxiliary optimization problem with the necessary
conditions as constraints, and solve for a feasible solution. However,
as they report in the paper, a feasible solution is not easy to find.
They need to relax the equality and inequality constraints to the
error being less than or equal to a bound, and let the bound shrink
to zero with iteration. In addition, this method requires user intervention:
sophisticated artificial constraints need to be added to help the
solver find a feasible solution, and according to the authors, some
solvers were sensitive to the objective function, i.e., under the
same constraints, the solver may deem a problem infeasible with one
objective function, but could find the optimal solution when given
an another objective function. So when the problem doesn't solve,
the user doesn't know whether it is because the SFE doesn't exist
or it is because he/she is not using the right objective function.

Here we base on the same idea and improve by simplifying their method
with the use of splines. In fact, their piecewise linear functions
could be seen as splines with free knots (the knots were decision
variables in their model), but with formal use of splines we can greatly
reduce the number of variables and constraints of the problem, which
in principle makes it easier to find a feasible solution, and faster
to find the optimal solution.

If $\{s_{i}\}_{i=1}^{m}$ form an SFE, then for any firm $i$, and
for any demand shock $\varepsilon$, the corresponding market clearing
price $p$ must solve the optimization problem 
\begin{align*}
\max_{p}\quad & Profit(p)=[D(p)+\varepsilon-\sum_{j\ne i}s_{j}(p)]p-C_{i}(D(p)+\varepsilon-\sum_{j\ne i}s_{j}(p))\\
\text{s.t.}\quad & 0\le D(p)+\varepsilon-\sum_{j\ne i}s_{j}(p)\le Cap_{i}.
\end{align*}
If $\{s_{i}\}_{i=1}^{m}$ are differentiable at the optimal price
$p$, then they must satisfy the following Karush-Kuhn-Tucker (KKT)
conditions:
\begin{align}
 & s_{i}(p)+(p-C_{i}'(s_{i}(p))-\lambda_{i}+\mu_{i})\cdot(D'(p)-{\textstyle \sum}_{j\ne i}s_{j}'(p))=0,\nonumber \\
 & {\textstyle \sum_{j}}s_{j}(p)=D(p)+\varepsilon,\nonumber \\
 & s_{i}(p)\le Cap_{i},\nonumber \\
 & s_{i}(p)\ge0,\nonumber \\
 & \lambda_{i}(Cap_{i}-s_{i}(p))=0,\nonumber \\
 & \mu_{i}s_{i}(p)=0,\nonumber \\
 & \lambda_{i},\mu_{i}\ge0,\label{eq:OriginalNC}
\end{align}
where $\lambda_{i}$ and $\mu_{i}$ are Lagrangian multipliers corresponding
to the capacity and non-negativity constraints, respectively. Replace
$\{s_{i}\}_{i=1}^{m}$ with their spline approximations $\{\hat{s}_{i}\}_{i=1}^{m}$
in (\ref{eq:OriginalNC}), and assemble (\ref{eq:OriginalNC}) for
all firms and a set of demand realizations $\varepsilon_{k}$, $k=1,\dots,N$.
Assume that $\{s_{i}\}_{i=1}^{m}$ are differentiable at the optimal
prices $p_{k}=p(\varepsilon_{k})$, then the KKT conditions become:
\begin{align}
 & \hat{s}_{i}(p_{k})+(p_{k}-C_{i}'(\hat{s}_{i}(p_{k}))-\lambda_{ik}+\mu_{ik})(D'(p_{k})-{\textstyle \sum_{j\ne i}}\hat{s}_{j}'(p_{k}))=0,\,\text{ for all }i,k,\nonumber \\
 & {\textstyle \sum_{j}}\hat{s}_{i}(p_{k})=D(p_{k})+\varepsilon_{k},\text{ for all }k,\nonumber \\
 & \hat{s}_{i}(p_{max})\le Cap_{i},\text{ for all }i,\nonumber \\
 & \hat{s}_{i}(p_{min})\ge0,\text{ for all }i,\nonumber \\
 & \lambda_{ik}(Cap_{i}-\hat{s}_{i}(p_{k}))=0,\text{ for all }i,k,\nonumber \\
 & \mu_{ik}\hat{s}_{i}(p_{k})=0,\text{ for all }i,k,\nonumber \\
 & \lambda_{ik},\mu_{ik}\ge0,\text{ for all }i,k,\label{eq:NCAssemble}
\end{align}
where $\hat{s}_{i}(p_{k})$ is expressed as $\hat{s}_{i}(p_{k})=B^{T}(p_{k})\beta_{i}$
in computation.

Under what conditions do $\beta_{i}$ , $\lambda_{ik}$ and $\mu_{ik}$
exist that satisfy (\ref{eq:NCAssemble})? The answer depends on what
splines we use and how we select knots. For example, if we use splines
of order 4, and if we set one knot at each price level, i.e., $\tau_{k}=p_{k}$,
then we are sure there exist $\beta_{i}$ , $\lambda_{ik}$ and $\mu_{ik}$
that satisfy (\ref{eq:NCAssemble}), given the SFE itself exists.
In fact, conditions (\ref{eq:OriginalNC}) are all about $\{s_{i}\}_{i=1}^{m}$
and their first order derivatives. If we have $B^{T}(p_{k})\beta_{i}=s_{i}(p_{k})$
and $B'^{T}(p_{k})\beta_{i}=s_{i}'(p_{k})$ for all $i$ and $k$,
then the $\beta_{i}$ and the original $\lambda_{ik}$ and $\mu_{ik}$
will automatically satisfy (\ref{eq:NCAssemble}). This is not difficult.
Since $\tau_{k}=p_{k}$, for any $i$, from $p_{k}$ to $p_{k+1}$,
$B^{T}(p)\beta_{i}$ is a single piece cubic polynomial. And $B^{T}(p_{k})\beta_{i}=s_{i}(p_{k})$,
$B^{T}(p_{k+1})\beta_{i}=s_{i}(p_{k+1})$, $B'^{T}(p_{k})\beta_{i}=s_{i}'(p_{k})$
and $B'^{T}(p_{k+1})\beta_{i}=s_{i}'(p_{k+1})$ place 4 constraints
that will determine the polynomial. A piecewise cubic Hermite interpolation
is a spline that satisfies these constraints. (See de Boor \cite{de2001practical}.)

More generally, if we are using B-splines, for example, then a $K$-knot
B-spline of order $M$ has $K+M$ coefficients. If we consider a price
range where $0<s_{i}<Cap_{i}$ for all $i$, thus all supply functions
are smooth and all $\lambda_{ik}$ and $\mu_{ik}$ are 0, then there
is only one equality constraint per firm per price level, i.e., the
first constraint in (\ref{eq:NCAssemble}). Therefore, if we place
only one price level between every two adjacent knots, then there
exist $\{\beta_{i}\}_{i=1}^{m}$ that satisfy (\ref{eq:NCAssemble}).
However, if we have more than one price level between some adjacent
knots, then a solution that satisfies (\ref{eq:NCAssemble}) may not
exist.

As in Anderson and Hu \cite{anderson2008finding}, we use an auxiliary
optimization problem to find a feasible solution. The problem here
is that $B^{T}(p)$ has no simple analytical expression, thus it can
hardly be evaluated by a solver. Fortunately, since $\{s_{i}\}$ are
optimal for all the values of $\varepsilon$, $\{\varepsilon_{k}\}$
do not have to be chosen to reflect the distribution of $\varepsilon$.
Hence, instead of selecting $\{\varepsilon_{k}\}$ and optimizing
$\{p_{k}\}$, we can fix $\{p_{k}\}$, and let $\{\varepsilon_{k}\}$
be the decision variables. Further, examining (\ref{eq:NCAssemble})
closely, one would find that $\{\varepsilon_{k}\}$ do not have to
appear as decision variables at all: they are simply determined by
$\varepsilon_{k}=B^{T}(p_{k})\sum_{j}\beta_{j}-D(p_{k})$. If there
is an $\varepsilon_{k}$ larger (smaller) than the upper (lower) bound
of the support of $\varepsilon$, it means that we have chosen a $p_{k}$
too large (too small) that is not needed for the supply functions.

In addition, feasible supply functions must be non-decreasing. Thus
we place the monotonicity constraint $\hat{s}_{i}(p_{k})\le\hat{s}_{i}(p_{k+1})$,
for all $i$ and $k$. However, with this new constraint, we are no
longer guaranteed to find $\beta_{i}$ , $\lambda_{ik}$ and $\mu_{ik}$
that satisfy (\ref{eq:NCAssemble}), i.e., there may not be a feasible
solution in the spline space, which means that we need to do relaxations.
We replace the ``$=0$'' constraints in \eqref{eq:NCAssemble} with
their absolute values less than or equal to $\rho$, where $\rho\ge0$.

The objective is simply ``minimize $\rho$'', thus there is no need
to use iteration as in \cite{anderson2008finding}.

The complete formulation is now as follows:
\begin{align}
\min_{\rho,\beta,\lambda,\mu}\quad & \rho\nonumber \\
\text{s.t.}\quad & {\textstyle \left|\hat{s}_{i}(p_{k})+(p_{k}-C_{i}'(\hat{s}_{i}(p_{k}))-\lambda_{ik}+\mu_{ik})(D'(p_{k})-\sum_{j\ne i}\hat{s}_{j}'(p_{k}))\right|}\le\rho,\text{ for all }i,k,\nonumber \\
 & \hat{s}_{i}(p_{k})\le\hat{s}_{i}(p_{k+1}),\text{ for all }i,k,\nonumber \\
 & \hat{s}_{i}(p_{max})\le Cap_{i},\text{ for all }i,\nonumber \\
 & \hat{s}_{i}(p_{min})\ge0,\text{ for all }i,\nonumber \\
 & \lambda_{ik}(Cap_{i}-\hat{s}_{i}(p_{k}))\le\rho,\text{ for all }i,k,\nonumber \\
 & \mu_{ik}\hat{s}_{i}(p_{k})\le\rho,\text{ for all }i,k,\nonumber \\
 & \lambda_{ik},\mu_{ik}\ge0,\text{ for all }i,k,\label{eq:AuxProb}
\end{align}
where $\hat{s}_{i}(p_{k})=B^{T}(p_{k})\beta_{i}$.

Ideally, we would hope that the optimal value of $\rho$ to be 0.
But in reality, it is often the case that there is not a solution
in the spline space that exactly fits all the conditions in (\ref{eq:NCAssemble}),
thus the optimal $\rho$ will be positive. However, due to the flexibility
of splines, there are functions in the spline space that ``almost''
fit (\ref{eq:NCAssemble}), i.e., the optimal $\rho$ will be small
(See examples in Section \ref{sec:NumericalExamples}).

We now show an asymptotic property of the approximation. Let the knots
be $p_{min}=\tau_{0}<\tau_{1}<\cdots<\tau_{N}=p_{max}$, and for simplicity,
let the controlled prices be $p_{k}=\frac{1}{2}(\tau_{k-1}+\tau_{k})$,
$k=1,\dots,N$. We will only show and prove the property for quadratic
splines, but it can be proved very similarly for splines of higher
orders.
\begin{thm}
\label{Thm:asympotic}Assume that $\{s_{i}\}_{i=1}^{m}$ form an equilibrium
and are differentiable at $\{\tau_{k}\}_{k=0}^{N}$ and at $\{p_{k}\}_{k=1}^{N}$
that are defined as above.%
\footnote{If not differentiable, choose slightly different $\{\tau_{k}\}_{k=0}^{N}$
and $\{p_{k}\}_{k=1}^{N}$. Remember that every $s_{i}(p)$ is continuously
differentiable only except at $p\in\left\{ C_{1}'(0),\dots,C_{m}'(0)\right\} \cup\left\{ p_{Cap_{1}},\dots,p_{Cap_{m}}\right\} $.%
} Let $\{\hat{s}_{i}\}_{i=1}^{m}$ be the optimal solution to (\ref{eq:AuxProb}),
among quadratic splines, where the knots are $\{\tau_{k}\}_{k=0}^{N}$
and prices are $\{p_{k}\}_{k=1}^{N}$, then every price $p\in[p_{min},p_{max}]$
will eventually satisfy the KKT conditions (\ref{eq:OriginalNC}),
as $\left|\tau\right|\rightarrow0$.\end{thm}
\begin{proof}
We first show that the KKT conditions will eventually be satisfied
at all the controlled points, i.e., the optimal value $\rho\rightarrow0$,
as $\left|\tau\right|\rightarrow0$. Then we show that the KKT conditions
will be satisfied at any point $p\in[p_{min},p_{max}]$.

For every $i$, let $\tilde{s}_{i}$ be a smooth approximation of
$s_{i}$, such that if $[\tau_{k-1},\tau_{k}]$ does not contain a
nonsmooth point of $s_{i}$, then $\tilde{s}_{i}(p)=s_{i}(p)$, for
all $p\in[\tau_{k-1},\tau_{k}]$; otherwise, we only require non-decreasingness
and $\tilde{s}_{i}(p)=s_{i}(p)$ at $p_{k}$. Note that as $\left|\tau\right|\rightarrow0$,
we will not have two adjacent intervals that both contain nonsmooth
points, and each interval will contain at most one nonsmooth point.

Let $I_{3}\tilde{s}_{i}$ be the quadratic interpolation of $\tilde{s}_{i}$,
such that $I_{3}\tilde{s}_{i}(p_{k})=\tilde{s}_{i}(p_{k})=s_{i}(p_{k})$
for all $k=1,\dots,N$. Thus the monotonicity constraint, the capacity
constraint and the nonnegativity constraint are automatically satisfied.

Since $\tilde{s}_{i}$ is smooth, we can use the property proved in
Marsden \cite{marsden1974quadratic} that for all $k$, 
\[
\left|I_{3}\tilde{s}_{i}'(p_{k})-\tilde{s}_{i}'(p_{k})\right|\le3\sup\left\{ |\tilde{s}_{i}'(p)-\tilde{s}_{i}'(p')|\text{: such that }|p-p'|\le\frac{\left|\tau\right|}{2}\right\} =O(\left|\tau\right|).
\]

Let $\mu_{ik}$ and $\lambda_{ik}$ be the same as they are in (\ref{eq:OriginalNC}),
so that the complementarity constraints are satisfied, and only the
first constraint in (\ref{eq:AuxProb}) will affect the optimal value
$\rho$.

Consider $[\tau_{k-1},\tau_{k}]$ that does not contain a nonsmooth
point of $s_{i}$. We have $\tilde{s}_{i}(p)=s_{i}(p)$ for all $p\in[\tau_{k-1},\tau_{k}]$.
The error of the first order condition given by $I_{3}\tilde{s}_{i}$
will be: 
\begin{align*}
 & \left|I_{3}\tilde{s}_{i}(p_{k})+(p_{k}-C_{i}'(I_{3}\tilde{s}_{i}(p_{k}))-\lambda_{ik}+\mu_{ik})(D'(p_{k})-{\textstyle \sum}_{j\ne i}I_{3}\tilde{s}_{j}'(p_{k}))\right|\\
= & \left|s_{i}(p_{k})+(p_{k}-C_{i}'(s_{i}(p_{k})))-\lambda_{ik}+\mu_{ik})(D'(p_{k})-{\textstyle \sum_{j\ne i}}I_{3}\tilde{s}_{j}'(p_{k}))\right|\\
= & \left|s_{i}(p_{k})+(p_{k}-C_{i}'(s_{i}(p_{k})))-\lambda_{ik}+\mu_{ik})(D'(p_{k})-{\textstyle \sum_{j\ne i}}s_{j}'(p_{k})+{\textstyle \sum_{j\ne i}}(s_{j}'(p_{k})-I_{3}\tilde{s}_{j}'(p_{k})))\right|\\
= & \left|(p_{k}-C_{i}'(s_{i}(p_{k})))-\lambda_{ik}+\mu_{ik})({\textstyle \sum_{j\ne i}}(s_{j}'(p_{k})-I_{3}\tilde{s}_{j}'(p_{k})))\right|\\
\le & (p_{k}-C_{i}'(s_{i}(p_{k})))-\lambda_{ik}+\mu_{ik})({\textstyle \sum_{j\ne i}}\left|s_{j}'(p_{k})-I_{3}\tilde{s}_{j}'(p_{k})\right|)\\
= & (p_{k}-C_{i}'(s_{i}(p_{k})))-\lambda_{ik}+\mu_{ik})({\textstyle \sum_{j\ne i}}\left|\tilde{s}_{j}'(p_{k})-I_{3}\tilde{s}_{j}'(p_{k})\right|)\\
\le & (p_{k}-C_{i}'(s_{i}(p_{k})))-\lambda_{ik}+\mu_{ik})({\textstyle \sum_{j\ne i}}O(\left|\tau\right|))\\
= & O(\left|\tau\right|),
\end{align*}
where in the first equality we replaced $I_{3}\tilde{s}_{i}(p_{k})$
with $s_{i}(p_{k})$ as they are equal, in the second equality we
added and subtracted ${\textstyle \sum_{j\ne i}}s_{j}'(p_{k})$, in
the third equality we removed $s_{i}(p)+(p-C_{i}'(s_{i}(p))-\lambda_{i}+\mu_{i})\cdot(D'(p)-\sum_{j\ne i}s_{j}'(p))$
as it equals 0, and in the fourth equality we replaced $s'_{j}(p_{k})$
with $\tilde{s}_{j}'(p_{k})$, because by construction $\tilde{s}_{j}=s_{j}$
on$[\tau_{k-1},\tau_{k}]$.

Now consider $[\tau_{k-1},\tau_{k}]$ that does contain a nonsmooth
point $p^{*}$ of $s_{i}$. There are 2 cases: $p^{*}<p_{k}$ and
$p^{*}>p_{k}$ (We assumed $p_{k}$ is a differentiable point).

For the case $p^{*}<p_{k}$, in the next interval $[\tau_{k},\tau_{k+1}]$,
for all $i=1,\dots,m$, (a) $s_{i}$ is smooth, (b) $\tilde{s}_{i}=s_{i}$,
and (c) $I_{3}\tilde{s}_{i}$ is a quadratic interpolation of $\tilde{s}_{i}$.
So we have (d) $s_{i}'(p_{k})=s_{i}'(\tau_{k})+O(\left|\tau\right|)$
(by a), (e) $I_{3}\tilde{s}_{i}'(\tau_{k})=\tilde{s}_{i}'(\tau_{k})+O(\left|\tau\right|)=s_{i}'(\tau_{k})+O(\left|\tau\right|)$
(by b and c), and (f) $I_{3}\tilde{s}_{i}'(p_{k})=I_{3}\tilde{s}_{i}'(\tau_{k})+O(\left|\tau\right|)$
(by c). Thus by d, e and f, $\left|s_{i}'(p_{k})-I_{3}\tilde{s}_{i}'(p_{k})\right|=O(\tau)$,
for all $i$.

Similarly, in the case $p^{*}>p_{k}$, we look at the previous interval.
So for all $i$, $s_{i}'(p_{k})=s_{i}'(\tau_{k-1})+O(\left|\tau\right|)$,
$I_{3}\tilde{s}_{i}'(\tau_{k-1})=\tilde{s}_{i}'(\tau_{k-1})+O(\left|\tau\right|)=s_{i}'(\tau_{k-1})+O(\left|\tau\right|)$,
and $I_{3}\tilde{s}_{i}'(p_{k})=I_{3}\tilde{s}_{i}'(\tau_{k-1})+O(\left|\tau\right|)$,
thus $\left|s_{i}'(p_{k})-I_{3}\tilde{s}_{i}'(p_{k})\right|=O(\tau)$.

Therefore, the error of the first order condition given by $I_{3}\tilde{s}_{i}$
will be: 
\begin{align*}
 & \left|I_{3}\tilde{s}_{i}(p_{k})+(p_{k}-C_{i}'(\hat{s}_{i}(p_{k}))-\lambda_{ik}+\mu_{ik})(D'(p_{k})-{\textstyle \sum_{j\ne i}}I_{3}\tilde{s}_{j}'(p_{k}))\right|\\
\le & (p_{k}-C_{i}'(s_{i}(p_{k})))-\lambda_{ik}+\mu_{ik})({\textstyle \sum_{j\ne i}}\left|s_{j}'(p_{k})-I_{3}\tilde{s}_{j}'(p_{k})\right|)\\
\le & (p_{k}-C_{i}'(s_{i}(p_{k})))-\lambda_{ik}+\mu_{ik})({\textstyle \sum_{j\ne i}}O(\left|\tau\right|))\\
= & O(\left|\tau\right|).
\end{align*}

Thus we can conclude that the best $\rho$ with $\{I_{3}\tilde{s}_{i}\}_{i=1}^{m}$
is $O(\left|\tau\right|)$. And since $\{I_{3}\tilde{s}_{i}\}_{i=1}^{m}$
is just one of the feasible approximations, the optimal value $\rho$
for (\ref{eq:AuxProb}) must be at least as good as $O(\left|\tau\right|)$.
Therefore $\rho\rightarrow0$, as $\left|\tau\right|\rightarrow0$.

The uniform convergence follows naturally due to uniform continuity:
Let $\hat{\lambda}_{i}(p)$ and $\hat{\mu}_{i}(p)$, $i=1,\dots,m$,
be linear interpolations of $\{\lambda_{ik}\}$ and $\{\mu_{ik}\}$
(so they are non-negative), i.e.,
\[
\hat{\lambda}_{i}(p)=\begin{cases}
\lambda_{ik}, & p=p_{k};\\
\frac{p_{k}-p}{p_{k}-p_{k-1}}\lambda_{i,k-1}+\frac{p-p_{k-1}}{p_{k}-p_{k-1}}\lambda_{ik}, & p_{k-1}<p<p_{k},
\end{cases}
\]
and
\[
\hat{\mu}_{i}(p)=\begin{cases}
\mu_{ik}, & p=p_{k};\\
\frac{p_{k}-p}{p_{k}-p_{k-1}}\mu_{i,k-1}+\frac{p-p_{k-1}}{p_{k}-p_{k-1}}\mu_{ik}, & p_{k-1}<p<p_{k}.
\end{cases}
\]
So the approximated supply functions $\{\hat{s}_{i}(p)\}$, and the
Lagrangians $\{\hat{\lambda}_{i}(p)\}$ and $\{\hat{\mu}_{i}(p)\}$
are uniformly continuous on $[p_{min},p_{max}]$. Also, the error
functions of the first order conditions 
\[
FOC_{i}(p)=\hat{s}_{i}(p)+(p-C_{i}'(\hat{s}_{i}(p))-\hat{\lambda}_{i}(p)+\hat{\mu}_{i}(p))(D'(p)-\sum_{j\ne i}\hat{s}_{j}'(p)),\; i=1,\dots,m
\]
and the error functions of the complementarity conditions
\[
LAM_{i}(p)=\hat{\lambda}_{i}(p)(Cap_{i}-\hat{s}_{i}(p)),\; i=1,\dots,m
\]
and 
\[
MU_{i}(p)=\hat{\mu}_{i}(p)\hat{s}_{i}(p),\; i=1,\dots,m
\]
are uniform continuous.

For any $p\in[p_{min},p_{max}]$, let $p_{\bar{k}}$ be the nearest
point to $p$ among $\{p_{k}\}$. Hence, for any $\epsilon>0$, as
proved above, there exists $\delta_{1}>0$, such that when $\left|\tau\right|<\delta_{1}$,
for all $i$, we have $\left|FOC_{i}(p_{\bar{k}})\right|<\frac{\epsilon}{2}$,
$\left|LAM_{i}(p_{\bar{k}})\right|<\frac{\epsilon}{2}$ and $\left|MU_{i}(p_{\bar{k}})\right|<\frac{\epsilon}{2}$.
Convergence at all $\{p_{k}\}$ are dominated by $\rho$, so $\delta_{1}$
is independent of $p_{\bar{k}}$. By construction, $0\le\hat{s}_{i}(p_{\bar{k}})\le Cap_{i}$.
Also, there exists $\delta_{2}>0$, such that when $\left|\bar{p}-\tilde{p}\right|<\delta_{2}$,
due to uniform continuity, we have $\left|FOC_{i}(\bar{p})-FOC_{i}(\tilde{p})\right|<\frac{\epsilon}{2}$,
$\left|LAM_{i}(\bar{p})-LAM_{i}(\tilde{p})\right|<\frac{\epsilon}{2}$,
$\left|MU_{i}(\bar{p})-MU_{i}(\tilde{p})\right|<\frac{\epsilon}{2}$
and $\left|\hat{s}_{i}(\bar{p})-\hat{s}_{i}(\tilde{p})\right|<\epsilon$.
Thus, for $\left|\tau\right|<\min(\delta_{1},\delta_{2})$, we must
have $\left|FOC_{i}(p)\right|<\epsilon$, $\left|LAM_{i}(p)\right|<\epsilon$,
$\left|MU_{i}(p)\right|<\epsilon$ and $-\epsilon\le\hat{s}_{i}(p_{\bar{k}})\le Cap_{i}+\epsilon$.
Therefore, the KKT conditions will eventually be satisfied uniformly
at all $p\in[p_{min},p_{max}]$, as $\left|\tau\right|\rightarrow0$.
\end{proof}
The first constraint in (\ref{eq:AuxProb}) from the KKT conditions
is just what the ODE system (\ref{eq:firstOrder}) says, thus Proposition
3 in \cite{holmberg2008supply} and Theorem \ref{Thm:asympotic} together
guarantee that in the limit situation the solution of (\ref{eq:AuxProb})
is a supply function equilibrium (when it exists).

In case the solution gets stuck at a local minimum, one can try replacing
the constraint $\hat{s}_{i}(p_{k})\le\hat{s}_{i}(p_{k+1})$ with $b_{i,t}\le b_{i,t+1}$,
if one uses B-splines. The constraint $b_{i,t}\le b_{i,t+1}$ is a
sufficient condition for non-decreasingness for B-splines, so using
it instead of the necessary condition $\hat{s}_{i}(p_{k})\le\hat{s}_{i}(p_{k+1})$
will reduce the space of feasible solutions, thus making the solution
less likely to fall into local minima. Also, one does not want to
over reduce the space, so quadratic splines are recommended, because
$b_{i,t}\le b_{i,t+1}$ is a necessary and sufficient condition for
quadratic splines to be non-decreasing. In addition, our experience
shows that although the optimal $\rho$ given by the pointwise monotonicity
constraint $\hat{s}_{i}(p_{k})\le\hat{s}_{i}(p_{k+1})$ is usually
smaller than that from the full monotonicity constraint $b_{i,t}\le b_{i,t+1}$,
the solution from the latter formulation is usually more robust than
the former (See Example \ref{exa:compare}). Thus it is always good
to consider using the full monotonicity constraint, even when local
minimum is not present.

We see that compared to the formulation given in Anderson and Hu (see
\cite{anderson2008finding} (14), (17) and (18)), formulation (\ref{eq:AuxProb})
has significantly less variables and constraints, so in principle
it is easier to find a feasible solution and faster to solve. It is
also user friendlier, as it does not require the user to adjust the
objective function and constraints when solving a problem. Since there
is no description in \cite{anderson2008finding} about in which cases
it is hard to find feasible solutions, we are unable to make a comparison,
but we have not experienced any difficulty in the many problems we
tested.

When solving the problem we selected finite $\{p_{k}\}$. If the firms
were able to choose $p$ continuously, they may be able to improve
the profit slightly. The improvement of firm $i$'s profit at $p_{k}$
can be approximated by $\left|Profit'(p_{k})\cdot(p-p_{k})\right|+O(\left|\tau\right|^{2})$.
If $\lambda_{ik}=0$ and $\mu_{ik}=0$, then $\left|Profit'(p_{k})\cdot(p-p_{k})\right|<\rho\left|\tau\right|=O(\left|\tau\right|^{2})$.
So the improvement shrinks to 0 as $\left|\tau\right|\rightarrow0$.

\section{Numerical Examples\label{sec:NumericalExamples}}

The following are a few examples demonstrating the use of the numerical
methods we introduced above. Without loss of generality, constant
terms of all the cost functions are set to 0, as they do not affect
the results. When solving problems with the general method, both IPOPT
\cite{wachter2006implementation} and CONOPT \cite{drud1994conopt}
are good choices for the solver. With their default settings, CONOPT
tends to give a slightly better solution in terms of optimality and
feasibility, while IPOPT is much faster.
\begin{example}
\label{exa:lse}In this example we use the least squares method to
find the equilibrium in a duopoly market. The two firms have linear
cost functions: $C_{1}(q)=10q$, $C_{2}(q)=15q$. Their capacities
are $Cap_{1}=80$ and $Cap_{2}=75$, respectively. The demand function
is $D(p)=-3p+\varepsilon$. 

We use natural cubic splines in the example, while B-splines work
fine, too. The knots are from 5 to 77 at step 9. The price levels
used for fitting the first order condition (\ref{eq:firstOrder})
are from 16 to 65 at step 0.5. As described by Theorem \ref{P: singularity},
the $\mathbb{B}$ matrix has 18 columns but the rank is 17, giving
us one degree of freedom, which covers all the potential solutions
when $\varepsilon_{min}$ is low enough. Figure \ref{Flo:ex1_1} plots
the solutions to (\ref{eq:firstOrder}) with different values of $t$.

From Figure \ref{Flo:ex1_1} it is easy to tell that Firm 1 will reach
the capacity first. A linear search gives that the maximum of $\hat{s}_{1}$
will equal to $Cap_{1}$ (at $p_{Cap_{1}}=31.65$) when $t=194.06$.
Therefore, the obtained splines with $t=194.06$ gives an approximation
of the SFE for $15<p<p_{Cap_{1}}$. When $p\ge p_{Cap_{1}}$, $s_{1}(p)=80$
and $s_{2}(p)=3(p-15)$ until $s_{2}$ reaches $Cap_{2}$ at $p=40$.
When $10\le p\le15$, $s_{1}(p)=3(p-10)$ and $s_{2}(p)=0$. Figure
\ref{Flo:ex1_2} shows the approximated SFE for $10\le p\le45$. 
\end{example}
\begin{figure}
\begin{centering}
\subfloat[Solutions with different values of $t$]{\includegraphics[width=0.4\textwidth]{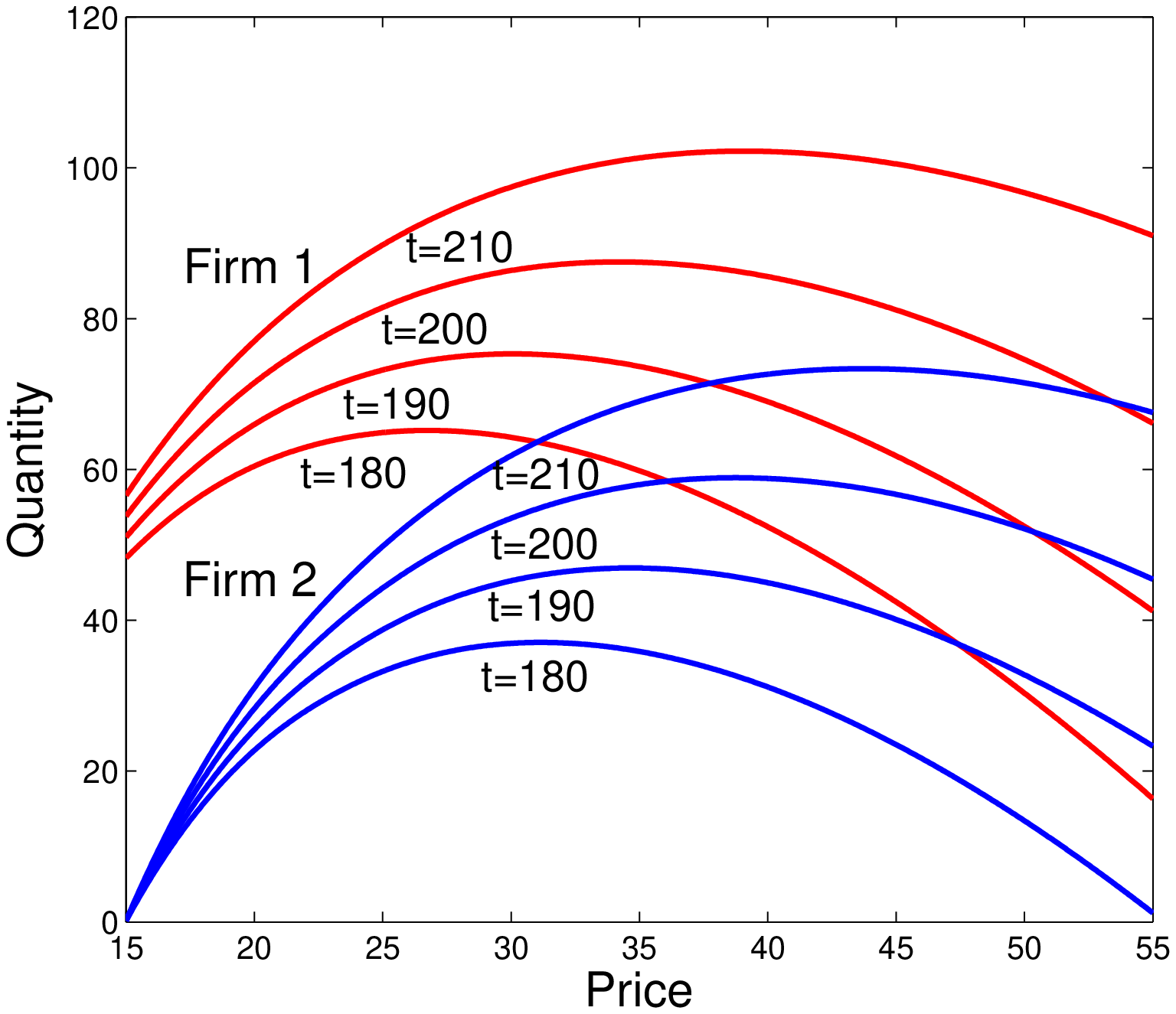}

\label{Flo:ex1_1}}\subfloat[Approximated SFE]{\includegraphics[width=0.4\textwidth]{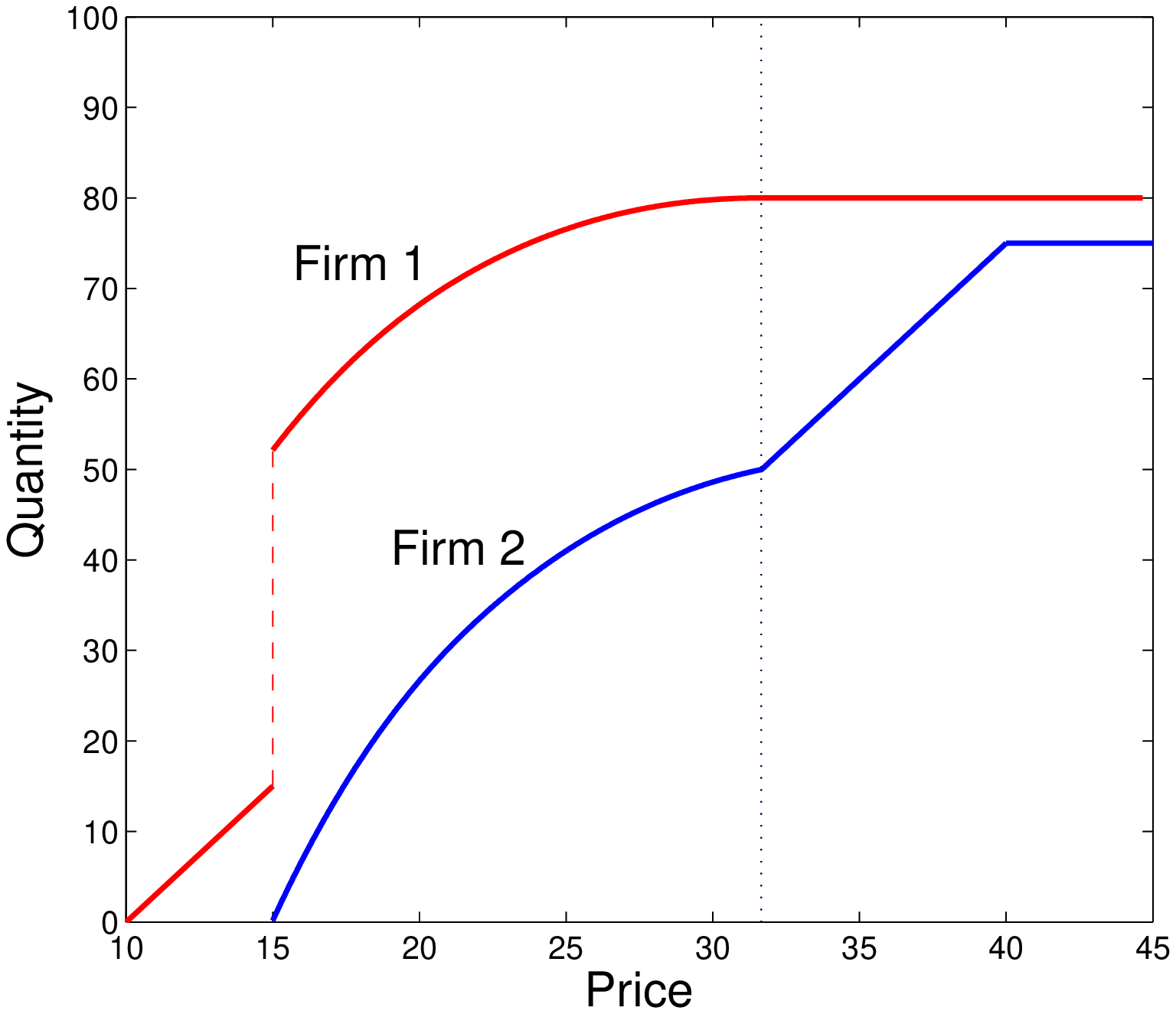}

\label{Flo:ex1_2}}
\par\end{centering}

\caption{Least squares approximation}
\end{figure}

\begin{example}
This time we find the SFE in Example \ref{exa:lse} with the general
method. All the splines we use with the general method are B-splines.
For this example, the knots are from 5 to 48 at step 0.05, and we
put one price level at the center of each knot interval.

We solved (\ref{eq:AuxProb}) with IPOPT using the full monotonicity
constraint, and obtained the optimal value $\rho=0.0048$, which is
sufficiently close to 0. Figure \ref{Flo:ex2} shows the spline approximation
of the SFE. We see that when the mesh is fine enough, splines are
quite capable at handling nonsmoothness and even discontinuities of
the functions.

Comparing Figure \ref{Flo:ex2} with Figure \ref{Flo:ex1_2}, we see
that both methods are able to find the equilibrium for markets of
asymmetric duopoly with constant marginal costs, and the solutions
are both of high precision. However, no matter in terms of the computational
time, or in terms of the tools, solving a least squares problem is
far easier than solving a highly nonlinear large scale optimization
problem. Thus for this type of problems, the specialized least squares
method is certainly more preferable.
\end{example}
\begin{figure}
\centering{}\includegraphics[width=0.4\textwidth]{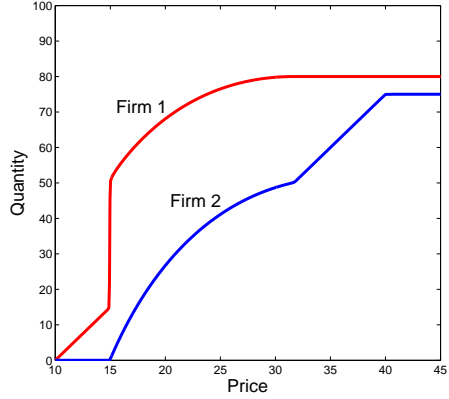}\caption{General method approximation}
\label{Flo:ex2}
\end{figure}

\begin{example}
\label{exa:compare}In this example we compare the effects of the
two monotonicity constraints, and the results with different mesh
sizes. The example is taken from Anderson and Hu \cite{anderson2008finding},
which has three firms with cost functions $C_{1}(q)=5q+0.8q^{2}$,
$C_{2}(q)=8q+1.2q^{2}$ and $C_{3}(q)=12q+2.3q^{2}$, and capacities
$Cap_{1}=11$, $Cap_{2}=8$ and $Cap_{3}=55$, respectively. The demand
function is $D(p)=-0.5p+\varepsilon$.

First we compare the pointwise monotonicity with the full monotonicity
constraints. The knots we use are from 5 to 54 at step 0.5, and we
put a price level at the center of each knot interval. We solve the
problem with CONOPT: The optimal value of $\rho$ is $1.6\times10^{-10}$
if we use the pointwise constraint, and it is $0.002$ if we use the
full constraint. Although the pointwise constraint gives a smaller
$\rho$, it does not necessarily mean that it is the better choice.
Figure \ref{Flo:ex3_1} is a comparison of the results at the low
price level, where images are magnified. Theoretically, we know that
when $5\le p<8$, Firm 1 is the monopoly, thus $s_{1}(p)=\frac{5}{18}(p-5)$,
and we also know that $s_{1}$ should have a jump at $p=8$. We see
that the solution given by the full constraint (\ref{Flo:Ex_3_1_b})
is closer to the true $s_{1}$ than the solution given by the pointwise
constraint (\ref{Flo:Ex3_1_a}) is. Therefore, despite a larger value
of $\rho$, the full monotonicity constraint is in fact more robust.

\begin{figure}
\centering{}\subfloat[pointwise monotonicity constraint]{\includegraphics[width=0.4\textwidth]{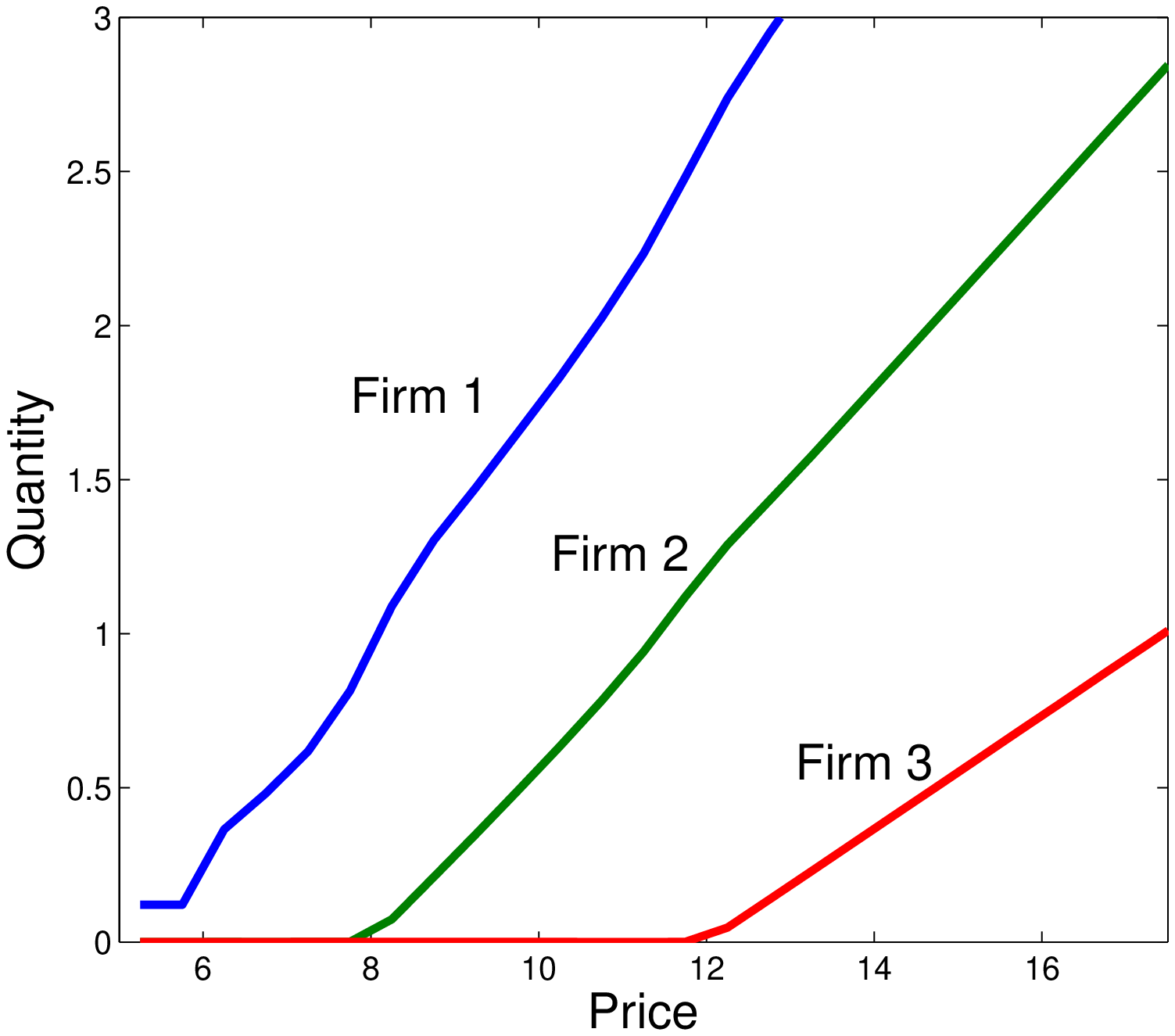}

\label{Flo:Ex3_1_a}}\subfloat[full monotonicity constraint]{\includegraphics[width=0.4\textwidth]{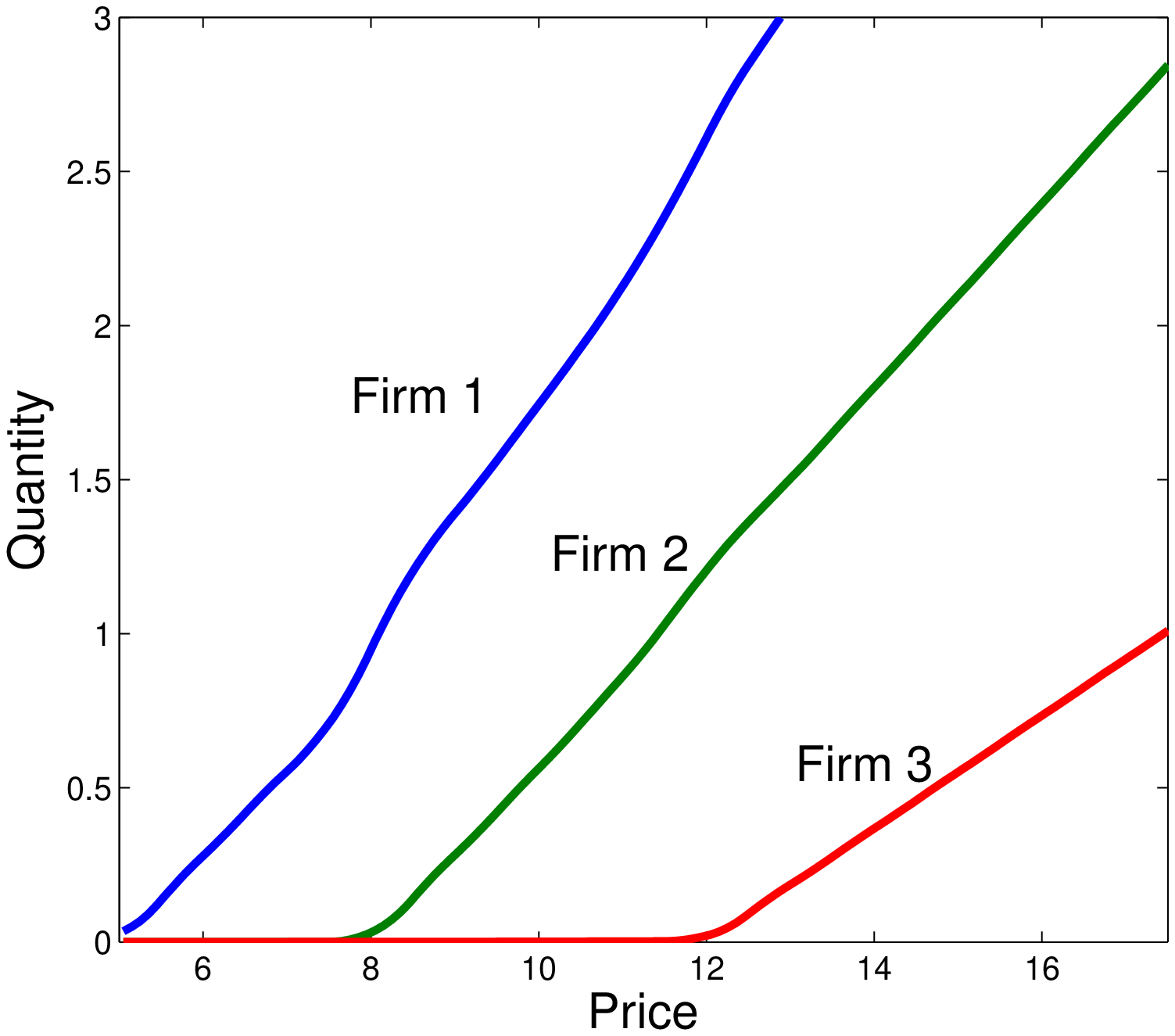}

\label{Flo:Ex_3_1_b}}\caption{Comparison of monotonicity constraints}
\label{Flo:ex3_1}
\end{figure}

We also see from Figure \ref{Flo:ex3_1} that although the full monotonicity
constraint gives a better approximation, it is still not close enough
to the true equilibrium. This is due to the fineness of the mesh,
and as we make the mesh finer, the result will be better. As an illustration,
we reduce the knot interval from 0.5 to 0.1, and again put one price
level at the center of each new knot interval. Keep the full monotonicity
constraint and solve the problem, CONOPT gives a new result with $\rho=0.00017$.
Figure \ref{Flo:ex3_2} compares the result of the finer approximation
(\ref{Flo:Ex3_2_b}) with the previous coarser approximation (\ref{Flo:Ex3_2_a}).
It is apparent that the precision has significantly improved. Figure
\ref{Flo:ex3_3} is the full plot of the finer approximation for $5\le p\le54$.

\begin{figure}
\begin{centering}
\subfloat[$\left|\tau\right|=0.5$]{\includegraphics[width=0.4\textwidth]{figure_ex3_2}

\label{Flo:Ex3_2_a}}\subfloat[$\left|\tau\right|=0.1$]{\includegraphics[width=0.4\textwidth]{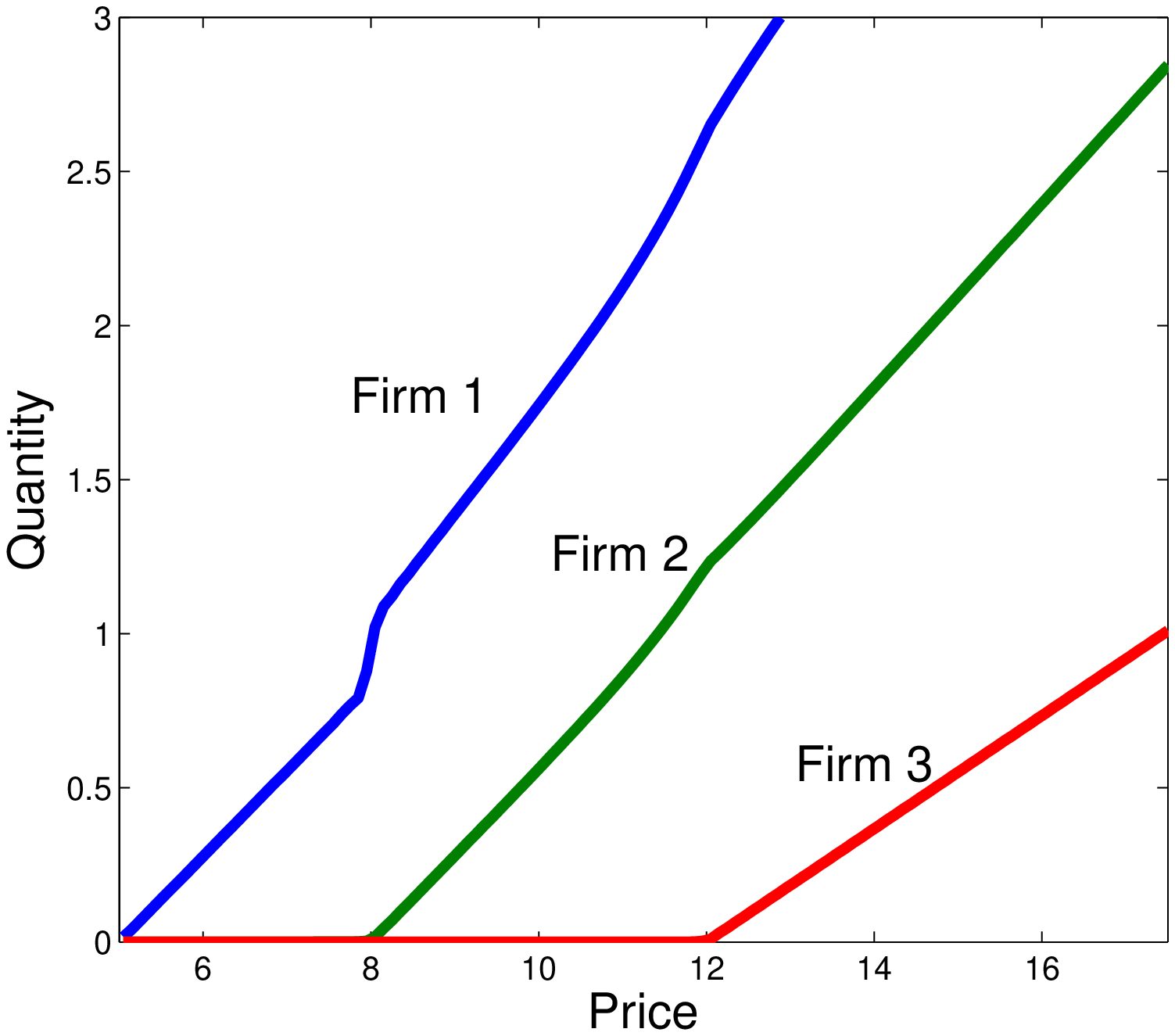}

\label{Flo:Ex3_2_b}}
\par\end{centering}

\caption{Comparison of mesh sizes}
\label{Flo:ex3_2}
\end{figure}

\begin{figure}
\begin{centering}
\includegraphics[width=0.6\textwidth]{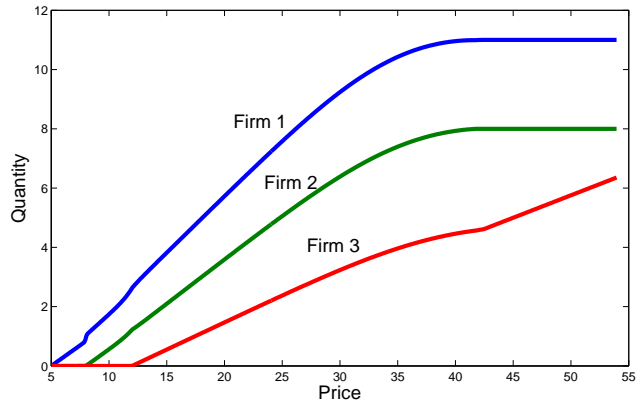}
\par\end{centering}

\caption{Approximation with $\left|\tau\right|=0.1$}
\label{Flo:ex3_3}
\end{figure}

For an even better approximation, one can always make the knot intervals
smaller. However, as the number of knots increases, the problem will
eventually become too large to solve. One way to improve the precision
while keeping the problem size tractable is to use the information
we obtained from a coarser approximation. From Figure \ref{Flo:ex3_3}
we can see that for $15\le p\le40$, the supply functions are very
smooth with little fluctuation, which means that a few pieces of quadratic
polynomials are good enough to approximate them. Therefore we can,
for example, set one knot at every 0.05 unit for $5\le p\le15$, and
at every 5 units for $15\le p\le40$, which allows us to save hundreds
of knots that would incur thousands of constraints.
\end{example}

\section{Conclusions}

One of the reasons why finding SFE has been so difficult is that the
supply functions do not have specific forms, so all the non-decreasing
functions (bounded by capacity) have to be considered. To find these
free-form functions, parameterization is almost inevitable, and splines,
due to their flexibility, are arguably the best way of parameterization
for the purpose of approximation. For duopolies with constant marginal
costs, we found that the first order conditions are linear in the
spline coefficients, allowing us to approximate the solutions of the
ODE system by solving a least squares problem. We proved that when
the demand can be sufficiently low with a positive probability, the
solution space of the least squares problem is exactly the solution
space of the ODE system. And since least squares problems are so easy
to solve, we can obtain solutions of high precision by using very
fine mesh, while still solving it fast. The solutions have a clean
form, which allows us to find the equilibrium easily by searching
for the supply functions that reach the capacities smoothly.

We also used splines to improve the general purpose method given in
Anderson and Hu \cite{anderson2008finding}. Both their original method
and our proposal should be equally accurate, but the use of splines
enabled us to significantly reduce the number of decision variables
and constraints used in the auxiliary NLP, thus making the problem
easier to solve, and without the need of human intervention.

The solutions of both the specialized and the general purpose methods
are proved to converge uniformly to the SFE. We also provided numerical
examples to demonstrate the use of these methods, and the solutions
are precise and reflect the theoretical properties of SFEs that we
developed throughout the paper.

\bibliographystyle{plunsrt}
\nocite{*}
\bibliography{SFE-spline}

\end{document}